\documentclass[11pt]{amsart}

\usepackage{amsmath}
\usepackage{amssymb}
\usepackage{amsfonts}
\usepackage{amsthm}
\usepackage{enumerate}
\usepackage{xcolor}
\newtheorem{theorem}{Theorem}[section]
\newtheorem{lemma}[theorem]{Lemma}
{Sublemma}
\newtheorem{proposition}[theorem]{Proposition}
\newtheorem{corollary}[theorem]{Corollary}

\theoremstyle{definition}

\newtheorem{example}[theorem]{Example}

\theoremstyle{remark}
\newtheorem{remark}[theorem]{Remark}

\numberwithin{equation}{section}

\def\R{{\mathbb R}}
\def\N{{\mathbb N}}

\def\C{{\mathbb C}}
\def\T{{\mathbb T}}
\newcommand{\e}{{\varepsilon}}

\newcommand{\wml}{w_{M,{\log}}}

\newcommand{\XL}{{\mathcal L}}
\newcommand{\XN}{{\mathcal N}}

\newcommand{\supp}{\operatorname{supp}}
\renewcommand{\Im}{\operatorname{Im}}
\renewcommand{\Re}{\operatorname{Re}}

\newcommand{\vanish}[1]{\relax}

\begin{document}

\title[$L^p$-tauberian theorems and $L^p$-rates]{$L^p$-tauberian theorems and $L^p$-rates for energy decay}
\author{Charles J.K. Batty}
\address{St. John's College, University of Oxford, Oxford OX1 3JP, United Kingdom}
\email{charles.batty@sjc.ox.ac.uk}

\author{Alexander Borichev}
\address{Institut de Math\'ematiques de Marseille, Aix Marseille Universit\'e, CNRS, Centrale Marseille, 39 rue F.~Joliot-Curie,
13453 Marseille, France}
\email{alexander.borichev@math.cnrs.fr}

\author{Yuri Tomilov}
\address{Institute of Mathematics, Polish Academy of Sciences,
\' Sniadeckich str.8, 00-956 Warsaw, Poland}
\email{ytomilov@impan.pl}
\subjclass{Primary 47D06; Secondary 34D05 34G10 35L05  40E05 44A10}

\date{\today}

\thanks{The research described in this paper was supported by the EPSRC grant EP/J010723/1. The second author was also partially supported by ANR FRAB. The third author was also partially supported by the NCN grant DEC-2014/13/B/ST1/03153 and by the EU grant  ``AOS'', FP7-PEOPLE-2012-IRSES, No 318910}

\keywords{Tauberian theorem,  energy decay, damped wave equation, Laplace transforms.}

\begin{abstract}
We prove $L^p$-analogues of the classical tauberian theorem of
Ingham and Karamata, and its variations giving rates of decay.
These
results are applied to derive $L^p$-decay of operator families arising
in the study of
the decay of energy for damped wave equations and local
energy for wave equations in exterior domains.
By constructing some
examples of critical behaviour we show that the $L^p$-rates of decay
obtained in this way
are best possible under our assumptions.
\end{abstract}
\maketitle

\section{Introduction}
One of the basic results in tauberian theory is a theorem due to Ingham \cite{In35} and Karamata \cite{Ka34}.  It has been used for elementary proofs of the prime number theorem, it has been a precursor for a number of famous results in function theory and operator theory such as theorems of Katznelson-Tzafriri type, and it still provides a link between many different applications of these theories. The result has found its way into many books and papers and has become a classic of modern tauberian theory; for a detailed discussion, see \cite{Ko04}.

One version of the theorem of Ingham and Karamata reads as follows \cite[Theorem 4.4.1]{ABHN01}, \cite[Theorem III.7.1]{Ko04}.

\begin{theorem}\label{ingham}
Let $X$ be a Banach space and let $f \in L^\infty(\mathbb R_+, X)$ be such that the Laplace transform $\widehat f$ admits an analytic extension to each point of $i\mathbb R$.
Then the improper integral $\int_{0}^{\infty}f(s)\, ds$ exists and equals $\widehat f(0)$.
\end{theorem}

Theorem \ref{ingham} can be equipped with rates as the next result shows.
For a continuous increasing function $M: \mathbb R_+\to [2,\infty)$, define
\begin{eqnarray}
M_{\log}(s) := M(s) [\log(1 +M(s)) + \log(1 + s)],\label{mlog}\\
\Omega_M:=\left \{\lambda \in \mathbb C: \Re\lambda > - \frac{1}{M(|\Im \lambda|)}\right \}. \label{omm}
\end{eqnarray}
The function  $M_{\log}$ is continuous, strictly increasing, with $\lim\limits_{s\to\infty} M_{\log}(s) = \infty$, so it has an inverse function $M_{\log}^{-1}$ defined on $[a,\infty)$ for some $a>0$.  The following was established in \cite{BaDu} (see \cite[Theorem 4.4.6]{ABHN01}).

\begin{theorem}\label{inghamrates}
Let $f \in L^{\infty}(\mathbb R_+,X)$, and assume that $\widehat f$ extends analytically to $\Omega_M$ and the extension satisfies
\begin{equation}  \label{resbound}
\|\widehat f(\lambda)\|\le M(|\Im \lambda |), \qquad \lambda \in \Omega_M.
\end{equation}
Let $M_{\log}$ be defined as above, and $c \in (0, 1)$. Then there exist positive numbers $C$ and $t_0$, depending
only on $\|f\|_{\infty}$, $M$  and $c$,  such that
\begin{equation} \label{Mlogest}
\Big\|\widehat f (0)-\int_{0}^{t}f(s)\,ds\Big\|\le \frac{C}{M^{-1}_{\log}(ct)}, \qquad t \ge t_0.
\end{equation}
\end{theorem}

There are versions of the Ingham--Karamata theorem allowing for a ``small'' set of singularities of $\widehat f$ on the imaginary axis \cite{ArBa88}.  The following is the simple case (for example, see \cite[Theorem 4.4.8]{ABHN01}).

\begin{theorem}\label{inghamsing} Let $f \in L^{\infty}(\mathbb R_+,X)$ be such that $\widehat f$ extends analytically to $i\mathbb R\setminus\{ia\}$, for some $a \in \R \setminus \{0\}$.  If
\begin{equation*} \label{singbound}
\sup_{t \ge 0} \Bigl\|\int_{0}^{t}e^{-i a s}f(s)\, ds\Bigr\|<\infty,
\end{equation*}
then
\begin{equation} \label{singlim}
\lim_{t \to \infty} \int_{0}^{t}f(s)\, ds=\widehat f(0).
\end{equation}
\end{theorem}

Tauberian theorems for functions are closely related to the study of asymptotics of  operator semigroups. In fact, the  tauberian theory, especially contour integral methods, laid the ground for various methods in stability theory of operator semigroups. Recently, the Ingham--Karamata theorem and the Korevaar--Neumann technique around it gave an impetus to the study of tauberian theorems with rates and their applications to partial differential equations (PDEs). They have been applied successfully to obtain rates of decay of classical solutions of abstract Cauchy problems \eqref{Cauchypr}. For recent applications of abstract results to the study of decay of solutions to PDEs, see, for example, the relevant references in \cite{BaChTo13}. In particular, energy decay for damped wave and wave equations was treated by resolvent methods in
\cite{AnLe12}, \cite{BaEnPrSchn06}, \cite{Bu98}, \cite{BuHi07}, \cite{Ch09}, \cite{ChSchVaWu13}, \cite{Le96},
\cite{LeRo97}, and \cite{LiRa07}.

To put tauberian theorems into the abstract framework of operator semigroups,
let us consider the abstract Cauchy problem
\begin{equation}\label{Cauchypr}
\left\{ \begin{array}{ll}
\dot{u} (t) = A u(t), & \quad t \ge 0 ,\\[2mm]
u(0) = x , & x\in X,
\end{array} \right.
\end{equation}
where $A$ is the generator of a
$C_0$-semigroup $(T(t))_{t \ge 0}$ on a Banach space $X$.
Since many PDEs arising from concrete models can be written in this form, the Cauchy problem \eqref{Cauchypr}  is a classical subject
of functional analysis having numerous applications to partial
differential equations.

Since the resolvent of the generator is often easier to compute than the semigroup, it is efficient to use the resolvent and its fine behaviour in the right half-plane to study the asymptotic behaviour of solutions to \eqref{Cauchypr} with particular regard to various kinds of stability in situations which arise from PDEs. A number of resolvent criteria for asymptotic and exponential stability are known in this context (see \cite[Chapter 5]{ABHN01}, \cite{ChTo07} and \cite{Ne96}).

In particular, starting from \cite{Le96} such methods have proved
to be successful in dealing with the damped wave equation
\begin{equation} \label{wave}
\begin{split}
 u_{tt} + a(x) u_t -\Delta u &= 0 \;\,\text{ in } \R_+ \times
 \mathcal M,\\
 u & = 0 \;\,\text{ in } \R_+ \times \partial \mathcal M, \\
 u(0,\cdot ) & = u_0 \text{ in } \mathcal M, \\
 u_t (0,\cdot ) & = u_1 \text{ in } \mathcal M.
\end{split}
\end{equation}
Here $\mathcal M$ is a smooth, compact, connected Riemannian
manifold, $\Delta$ is the Laplace--Beltrami operator on $\mathcal
M$,  and $a \in C^{\infty}(\mathcal M)$, $a \geq 0$. The energy
$E(u,t)$ of the solution $u$ to the problem \eqref{wave} is
defined by
\begin{equation*}
E(u,t)=\frac{1}{2}\left(\|\nabla u(t)\|^2_{L^2(M)}+\|u_t(t)\|^2_{L^2(M)}\right).
\end{equation*}
If  $a$ is strictly positive on an open subset of $\mathcal M$,
then $E(u,t)$  decays to zero as $t \to \infty$. One of the
primary tasks in the study of \eqref{wave} is to quantify the rate
of energy decay, that is to determine $r(t)$ such that
\begin{equation}\label{rateen}
E(u,t) \le r(t)^2 E(u,0)
\end{equation}
for all big enough $t$ and for all initial values $u_0,u_1$ from appropriate spaces. Let us
assume, for definiteness, that $\partial \mathcal M \neq
\emptyset$. Then the wave equation can be rewritten as a first
order Cauchy problem \eqref{Cauchypr} in the Hilbert space $X :=
H^1_0(\mathcal M)\times L^2(\mathcal M)$ with $A$ given  by
\begin{equation}\label{defA1}
A := \begin{pmatrix} O & I\\
                   \Delta & -a\end{pmatrix}
\end{equation}
on the domain
\begin{equation}\label{defA0}
 D(A)= (H^2(\mathcal M)\cap H^1_0(\mathcal M))\times H^1_0(\mathcal
 M).
\end{equation}
 The operator $A$ is invertible, and it generates a
non-analytic contraction semigroup $(T(t))_{t\geq 0}$ such that
$$
c\|T(t) x\|_X^2 \le E(u,t) \le C \|T(t) x\|_X^2
$$
where $x = (u_0,u_1)$ and $c,C$ are positive constants.  Thus any estimate of the rate of decay for
the semigroup is an estimate for the decay of the energy of the
system.  More specifically, let $r : [2,\infty) \to (0,\infty)$ be
a continuous decreasing function.  Then, following \eqref{rateen},
we will say that the equation \eqref{wave} is  {\it stable at rate
$r(t)$} if $E(u,t)^{1/2} \le C r(t) \|A (u_0, u_1)\|$ for all
$(u_0,u_1) \in D(A)$, in other words if $\|T(t)A^{-1}\| \le
Cr(t)$.

Moreover, an extensive line of research stemming from classical
works by Lax, Phillips, Melrose, Morawetz, Sjostrand, Strauss and
many others concerns the {\it local} energy decay of solutions to
wave equations in exterior domains. This situation arises when  $a
\equiv0$ in \eqref{wave} and $M=\R^n \setminus \overline K$, where
$K$ is a bounded domain with smooth boundary called an {\it
obstacle}. In abstract terms, one is led to the study of the norm decay of
operator families
\begin{equation}\label{utt}
U_{T_1, T_2}(t):=T_1 U(t) T_2,
\end{equation}
where $T_1$  and $T_2$ are the operators of multiplication by cut-off functions,
and $(U(t))_{t \in \mathbb R}$ is the unitary group
on $X$ generated by the operator matrix
\begin{equation*}
G:= \begin{pmatrix} O & I\\
                   \Delta & O\end{pmatrix}
\end{equation*}
with $D(G)=D(A)$.  From the enormous number of papers on local
energy decay we just mention the pioneering work \cite{Bu98}, and
\cite{BoPe06} and \cite{Ch09}, as samples where resolvent and
Laplace transform techniques led to efficient estimates for the
decay of $U_{T_1, T_2}$.  A good historical account of the decay of
local energy can be found in \cite{Bu98} and \cite{Ch09}.

To establish one of the links between tauberian and semigroup theories,
note that if $(T(t))_{t \ge 0}$ is a bounded $C_0$-semigroup on a Banach space $X$
with invertible generator $A$,
then setting $f(t)=T(t)x$ for $t\ge 0$, we infer that $\widehat f$ has a holomorphic
extension at $0$ and for any regular point $\mu$ of $A$,
\begin{eqnarray*}
T(t)R(\mu, A)x&=&R(\mu, A)x + AR(\mu, A)\int_{0}^{t}f(s)\, ds\label{link}\\
&=&AR(\mu, A)\left(\int_{0}^{t}f(s)\, ds -\widehat f(0)\right).\nonumber
\end{eqnarray*}
Moreover, by the resolvent equation, the rate of decay of $T(t)R(\mu, A)$ is independent of $\mu$ up to a multiplicative constant.
Thus a tauberian theorem for functions such as Theorem \ref{ingham} or \ref{inghamrates} above yields a corresponding statement for the decay of classical solutions (and related operator families) of the abstract Cauchy problem \eqref{Cauchypr}.  In particular, the following result can be obtained in this way (see \cite[Theorem 4.4.14]{ABHN01}).

\begin{theorem}\label{semigrouprates} Let $A$ be the generator of a bounded $C_0$-semigroup $(T(t))_{t \ge 0}$ on a Banach space $X$, and let $\mu \in \rho(A)$. Then the following are equivalent:
\begin{enumerate}[\rm(i)]
\item $\sigma(A)\cap i\mathbb R$ is empty,  \label{sgr1}
\item $\lim_{t \to \infty} \|T(t)R(\mu,A)\|=0$.
 \end{enumerate}
Moreover, if \eqref{sgr1} holds,
\begin{eqnarray*}
M(s) := \sup \{\|(ir+A)^{-1}\| : |r| \le s\}, \qquad s\ge 0,
\end{eqnarray*}
the function $M_{\log}$ is defined by \eqref{mlog}, and $M^{-1}$ is any right inverse for $M$, then there exist positive constants $C,C',c,c'$ such that
\begin{equation} \label{genestintro}
\frac{c'} {M^{-1}(C't)} \le \|T(t)A^{-1}\| \le \frac{C}
{M_{\log}^{-1}(ct)}
\end{equation}
for all sufficiently large $t$.
\end{theorem}

A natural and important question is whether the upper bound in \eqref{genestintro} can be improved to match the lower bound.
This is indeed possible in some cases, where $M$ and $X$ are particularly tractable, as the following result from \cite{BoTo10} shows.

\begin{theorem}\label{borto}
Let $(T(t))_{t \ge 0}$ be a bounded $C_0$-semigroup on a Hilbert space
$X$, with generator $A$.  Assume that $\sigma(A)\cap i\mathbb R$
is empty, and fix  $\alpha >0$.  Then the following statements
are equivalent:
\begin{enumerate}[\rm(i)]
\item   $\|R(is,A)\|= {\rm O}(|s|^\alpha), \qquad |s| \to \infty$.
\item  $\|T(t)A^{-1}\|= {\rm O}(t^{-1/\alpha}), \qquad t \to \infty$.
\end{enumerate}
\end{theorem}
 Theorem \ref{borto} has been shown to be very useful in a variety of different contexts. Since our presentation concentrates mainly on the equations of wave type, we restrict ourselves to mentioning the recent papers \cite{AnLe12}, \cite{BurqZuily} and \cite{LeautaudLerner}.
More general, and optimal or close to optimal, versions of Theorem \ref{borto} with the bound $|s|^\alpha$ in \eqref{sgr1} replaced by a bound of the form $|s|^\alpha \ell(|s|)$, where $\ell$ is a slowly varying function,
(for example, the bounds $|s|^\alpha (\log|s|)^{\pm \beta}$) were obtained recently in \cite{BaChTo13}.

However, it was shown in \cite{BoTo10} that the estimates in \eqref{genestintro} are sharp for polynomial rates on Banach spaces, and the logarithmic gap between $M_{\log}$ and $M$ is unavoidable in that case. Moreover, for some $M$, for example $M(s)=(1+\log|s|)^\alpha$ where $0<\alpha<1$, an estimate of the form $\|T(t)A^{-1}\|\le C/M^{-1}(ct)$ cannot hold even on Hilbert spaces (see \cite[Proposition 5.1]{BaChTo13}).

The Ingham--Karamata theorem and its variants in the literature, and their operator counterparts, concern bounded functions and bound\-ed operator semigroups.   The aim of the present paper is to provide similar theorems in the framework of $L^p$-spaces, with the case $p =\infty$ reproducing the earlier results. Such results have not been available in the
literature so far.  We show that our theorems are best possible,
at least in several important cases.  Apart from the contribution to tauberian theory, the paper gives several applications of its function-theoretic framework to operator models of PDEs in general and to damped wave equations in particular.  While the consequences of this approach for asymptotics of operator semigroups are rather trivial, this nevertheless leads to new results on asymptotic behaviour of the cut-off operator families defined in \eqref{utt}, and thus on decay of local energy for wave and similar equations.  There are many results on decay of local energy in both abstract and concrete settings. However, our approach to measuring the size of local energy in the $L^p$-sense seems to be new, and we are not aware of any papers in that direction.  This approach leads also to new information on energy decay for the damped wave equation as well, see Theorem \ref{Hilbert} and Example \ref{ex64}.

In particular, our new point of view provides the following $L^p$-version of Theorem \ref{inghamrates} (see Theorem \ref{x1}).

\begin{theorem}  \label{x0}
Let $f \in L^p(\mathbb R_+,X)$, where $1 \le p \le \infty$, and assume that $\widehat f$ extends analytically to $\Omega_M$ and satisfies
\begin{equation} \label{resboundp}
\|\widehat f(\lambda)\|\le M(|\Im \lambda |), \qquad \lambda \in \Omega_M.
\end{equation}
Then there exists $c>0$, depending only on $p$,  such that the function
\begin{equation*} \label{Mlogestp}
t \mapsto M_{\log}^{-1} (ct) \Big(\widehat f (0)-\int_{0}^{t}f(s)\,ds \Big)
\end{equation*}
belongs to $L^p(\R_+,X)$.
\end{theorem}

Theorem \ref{x0} remains true when \eqref{resboundp} is weakened somewhat (see Theorem \ref{x1}).

The paper is organised as follows. In Section~\ref{se2} we set up our notation and terminology and prove several auxiliary estimates for the function $M_{\log}$ (given by \eqref{mlog}) which will  frequently be used in the sequel.
The $L^p$-version of the Ingham--Karamata tauberian theorem is presented in Section~\ref{se3}. In Section~\ref{se4}, we prove a weighted $L^p$-version of the Ingham--Karamata theorem, Theorem~\ref{x1}, with the weight determined in terms of growth bounds on Laplace transforms in appropriate regions.
In Theorem~\ref{lpsing} we give an $L^p$-version of Theorem~\ref{inghamsing} in a quantified form, which is a new feature even for $p = \infty$.
Section~\ref{se5} provides an auxiliary function-theoretical construction showing that Theorem~\ref{x1} is sharp, at least for polynomial and logarithmic rates.
In Proposition~\ref{prop}, we construct certain complex measures with a good control of their Laplace transforms and of related functions, while in Theorem~\ref{example}, we use the Laplace transforms of these measures as building blocks to give an explicit example illustrating optimality of Theorem~\ref{x1}.
Operator-theoretical applications of our $L^p$ tauberian theorems to orbits of $C_0$-semigroups and cut-off semigroups as in \eqref{utt} are contained in Section~\ref{sect6}.
In this section, we prove two abstract theorems~\ref{energydecayabs} and \ref{Hilbert}.
Moreover, as a consequence of Theorem~\ref{Hilbert},  we obtain a new result, Corollary~\ref{c65}, on the asymptotics of solutions to damped wave equations on manifolds.
Theorems~\ref{xyz} and \ref{xyzA} given in Section~\ref{se7} show that our abstract results are optimal as well.

\section{Preliminaries}\label{se2}

Throughout this paper, $X$ will be a complex Banach space, and sometimes it will be a Hilbert space. We let ${\mathcal L}(X)$
denote the space of all bounded linear operators on Banach space $X$, and
the identity operator will be denoted by $I$. If $A$ is an (unbounded) linear
operator on $X$, we denote the domain of $A$ by $D(A)$,  the spectrum of $A$ by $\sigma(A)$, the resolvent set by $\rho(A)$,
and the resolvent of $A$ by $R(\lambda, A) := (\lambda-A)^{-1}$.

We shall  write
$\mathbb C_+ :=  \{\lambda \in \mathbb C : \operatorname{Re} \lambda  > 0\}$, $\mathbb C^+:=\{\lambda \in \mathbb C: \operatorname{ Im} \lambda >0\}$, $\C_- :=  \{\lambda \in \mathbb C : \operatorname{Re} \lambda  < 0\}$,
$\mathbb R_+ := [0,\infty)$, and $\mathbb R_- := (-\infty,0]$. The characteristic function of a set $E \subset \mathbb C$ will be denoted by $\chi_E$.  We shall use $C$ and $c$ to denote (strictly) positive constants, whose values may change from place to place.

We shall consider many functions $f$ defined on $\R_+$ with values in $(0,\infty)$, $\C$ or a Banach space $X$.  The Laplace transform of a measurable Laplace transformable function $f$ will be denoted by $\widehat f$.  Then $\widehat f$ is holomorphic on a right half-plane (usually $\C_+$ in this paper).  We shall use the same notation $\widehat f$ for any holomorphic extension to a larger connected open set.  For functions $f,g$ with values in $(0,\infty)$, the notation $f \asymp g$ means that $c g \le f \le C g$ for some positive $c,C$, while $f \sim g$ means that $\lim_{t\to\infty} f(t)/g(t) = 1$.

We shall use certain weighted $L^p$-spaces on $\R_+$.  Let $w : \R_+ \to (0,\infty)$ be a measurable function.  For a function $g : \R_+ \to X$, we shall write that $g \in L^p(\R_+,w,X)$ to mean that $g \cdot w \in L^p(\R_+,X)$.  We shall write
$$
\|g\|_{L^p(\R_+,w,X)} = \|g \cdot w\|_{L^p(\R_+,X)}.
$$
Our weights $w$ will always be increasing functions, so we may refer to such a statement  as saying
that $1/w$ is an {\it $L^p$-rate} of decay for $g$.  The precise form of $w$ on any interval $[0,a]$ is unimportant in such statements.

 Throughout this paper as in the Introduction, we shall let $M$ be a continuous increasing function from $\R_+$ to $[2,\infty)$, and $M_{\log}$ and $\Omega_M$ will be defined by  \eqref{mlog} and \eqref{omm}.  Thus $M_{\log}$ is continuous and strictly increasing on $\R_+$, with an inverse function $M_{\log}^{-1}$ defined on $[M_{\log}(0), \infty)$.   Since we are interested only in long-time asymptotic behaviour, the values of $M$ on any interval $[0,s_0]$ will be unimportant.

The conclusion of Theorem \ref{inghamrates} establishes that $w$ is an $L^\infty$-rate of decay for $g$, where
$$
 g(t) = \widehat f(0) - \int_0^t f(s) \, ds, \qquad  w(t) = M_{\log}^{-1}(ct). \qquad
$$
Here the weight $w$ is not defined on an interval of the form $[0,t_0]$.  This is unimportant for the discussion of $L^p$-rates, but for definiteness we define a weight $\wml$ on $\R_+$ by
$$
w_{M,\log}(t) = \begin{cases}  M_{\log}^{-1}(t)  \qquad &t \ge M_{\log}(1), \\
1 &0\le t\le M_{\log}(1).  \end{cases}
$$
In Theorem \ref{inghamrates} and in the main results of this paper, the $L^p$-rates of decay correspond to weights of the form $w(t)= \wml(kt)$ where $k$ is a constant depending on certain parameters.  For many functions $M$, one has $\wml(kt) \asymp \wml(t)$, so such rates of decay are independent of $k$.

Now we establish a few properties of the functions $M_{\log}$ and $\wml$.  We write $R(t) = \wml(t)$. Then, for $t\ge M_{\log}(1)$,
\begin{align*}
t = M_{\log} (R(t)) &= M(R(t)) \Big(\log(1+R(t)) + \log \big(1 + M(R(t))\big) \Big) \\
&\ge M(R(t)) \log \big( 1 + M(R(t))\big) .
\end{align*}
Since the inverse of $s \mapsto s \log(1+s)$ is asymptotically equivalent to $t \mapsto t(\log t)^{-1}$ as $t\to\infty$, it follows that
\begin{equation} \label{rmint}
 M(R(t)) \le \frac{C t}{\log t}
\end{equation}
for all sufficiently large $t$, and then for all $t\ge2$.  If $M(s) \ge \kappa s^\alpha$ for some $\kappa,\alpha>0$, we also have
$$
t  \le C M(R(t)) \log \big( 1 + M(R(t))\big)
$$
and then
\begin{equation} \label{rmint2}
 M(R(t)) \ge \frac{c t}{\log t}
\end{equation}
for all sufficiently large $t$, and then for all $t\ge2$.

The estimate \eqref{rmint2} may fail if $M$ grows slowly, for example if $M(s) = \log(2+s)$.  However there is an alternative estimate.  Let $\alpha>0$ and take $\e$ with $0 < \e < \min(\alpha,1)$.  For all sufficiently large $t$,
\begin{align*} \label{rmint3}
R(t)^\alpha M(R(t)) &\ge  R(t)^{\alpha-1} \big[R(t) M(R(t)) \log\big(R(t)M(R(t))\big)\big]^{1-\e} \\
&\ge c R(t)^{\alpha-\e} t^{1-\e} \ge ct^{1-\e}. \nonumber
\end{align*}
It follows that
\begin{equation} \label{rmint4}
\int_0^\infty R(t)^{-\alpha} M(R(t))^{-\beta}\, dt < \infty
\end{equation}
if $\alpha>0$, $\beta>1$.

\section{The Ingham--Karamata theorem in the $L^p$-setting}\label{se3}

First we need a very simple auxiliary estimate.
\begin{lemma}\label{auxil}
One has
\begin{equation}\label{exponent}
\int_{-\pi/2}^{\pi/2}e^{-t\cos\theta}\cos\theta\, d\theta \le \frac{C}{1+t^2}, \qquad t \ge 0.
\end{equation}
\end{lemma}

\begin{proof}
It suffices to note that
\begin{eqnarray*}
\int_{-\pi/2}^{\pi/2}e^{-t\cos\theta}\cos\theta\, d\theta&=&2\int_{0}^{\pi/2}e^{-t\sin\theta}\sin\theta\, d\theta \\
&\le& 2\int_{0}^{\pi/2}e^{-{2t\theta}/{\pi}} \theta\, d\theta < \frac{\pi^2}{2t^2}.
\end{eqnarray*}
\end{proof}

For $s \in \R$ and $y>0$, let  $P_y(s)=\frac{y}{\pi(s^2+y^2)}$ be the Poisson kernel for the upper half-plane.  For $f \in L^p(\mathbb R)$ where $1 \le p \le \infty$, the Poisson integral is defined by
\begin{equation*}
(P_y*f)(x)=\frac{1}{\pi} \int_{\mathbb R}\frac{y}{s^2+y^2}f(x+s)\, ds.
\end{equation*}
Recall that $\|P_y*f\|_{L^p}\le \|f\|_{L^p}$ for every $y>0$ \cite[Section VI.B]{Koo2}.

The following statement is our $L^p$-version of the Ingham--Karamata Theorem \ref{ingham}.

\begin{theorem}\label{lpingham}
Let $f \in L^p(\mathbb R_+, X)$, where $1 \le p \le \infty$.  Assume that $\widehat f$ admits an analytic extension to a neighbourhood of $0$ in $\C$, and let
$$
g(t)=\widehat f(0) -\int_{0}^{t} f(s) \, ds.
$$
Then $g \in L^p (\mathbb R_+, X)$.  If $p<\infty$, then  $g \in C_0(\mathbb R_+, X)$ and $g
\in L^r(\mathbb R_+,X)$ whenever $p \le r \le \infty$.
\end{theorem}

\begin{proof}
Without loss of generality we can assume that $\widehat f(0)=0$. Otherwise, we may consider $f -\chi_{[0,1]}\widehat f(0)$ instead of $f$.
By assumption $\widehat f$ extends analytically to a simply connected domain $\Omega \supset \C_+ \cup \{0\}$.  Take any integer $n\ge 2$, and $R>0$ such that $[-iR,iR] \subset \Omega$.  By the Cauchy integral formula, we have
\begin{equation} \label{Ciform}
g(t)=
\frac 1{2\pi i}\int_\gamma
\Bigl(1+\frac {z^2}{R^2}\Bigr)^n \Bigl(\widehat f(z) -\int_{0}^{t}
e^{-zs}f(s) \, ds\Bigr)e^{zt}\,\frac{dz}{z}, \quad t \ge 0,
\end{equation}
for any contour $\gamma\subset\Omega$ around $0$.
By deforming the contour to include a semi-circle of radius $R$ in $\C_+$ and then using the fact that the function $z\mapsto \int_{0}^{t}
e^{-zs}f(s)\, ds$ is entire and the function $z \mapsto \widehat f(z)/ z$ is analytic on $[-iR,iR]$, we infer that
\begin{multline}
2\pi \|g(t)\| \le
\Bigl\|\int_{\gamma_1} \Bigl(1+\frac {z^2}{R^2}\Bigr)^n \Bigl(\widehat
f(z) -\int_{0}^{t} e^{-zs}f(s) \, ds\Bigr)
e^{zt}\,\frac{dz}{z}\Bigr\|\\
 \null + \Bigl\|\int_{\gamma_2} \Bigl(1+\frac {z^2}{R^2}\Bigr)^n
\Bigl(\int_{0}^{t} e^{-zs}f(s) \, ds\Bigr)
e^{zt}\,\frac{dz}{z}\Bigr\|+ \Bigl\|\int_{I} \Bigl(1+\frac
{z^2}{R^2}\Bigr)^n \widehat f(z)
e^{zt}\,\frac{dz}{z}\Bigr\| \\
=: J_1(t)+J_2(t)+J_3(t), \label{sumj}
\end{multline}
where  $\gamma_1=R\mathbb T\cap \mathbb C_+$, $\gamma_2=R\mathbb T\cap \mathbb C_-$, and $I=[-iR,iR]$.

We estimate each of $J_1(t)$, $J_2(t)$ and $J_3(t)$ separately.
To estimate $J_1(t)$ we observe that for $z = Re^{i\theta}$,
$$
\left| 1 + \frac{z^2}{R^2} \right| = {2|\cos\theta|}.
$$
Hence by \eqref{exponent},
\begin{eqnarray} \label{j1est}
J_1(t) &\le&
C \int_{-\pi/2}^{\pi/2} \Bigl(\int_{0}^\infty
e^{-Rs\cos\theta}\|f(s+t)\|\,ds\Bigr)\, \cos\theta\,d\theta
\\
&\le& C \int_{0}^\infty \frac{ \|f(s+t)\|}{R^2s^2+1} \,ds \nonumber
\\
&\le&  \frac{C}{R} \left(P_{1/R}*h \right)(t),  \nonumber
\end{eqnarray}
where
\[
h(s)=\begin{cases} \|f(s)\|, &s\ge0, \\ 0, &s<0. \end{cases}
\]
Thus,
\begin{equation}\label{j1int}
\| J_1\|_{L^p}  \le \frac{C}{R}  \|f\|_{L^p}.
\end{equation}

We can estimate $J_2(t)$ in a similar way:
\begin{eqnarray} \label{j2est}
J_2(t) &\le&
C \int_{-\pi/2}^{\pi/2} \Bigl(\int_{0}^t e^{-Rs\cos\theta}\|f(t-s)\|\,ds\Bigr)\, \cos\theta\,d\theta
\\
& \le& C\int_{0}^t
\frac{\|f(t-s)\|}{R^2s^2+1} \,ds \nonumber
\\
&\le& \frac{C}{R}  \big(P_{1/R}* h\big)(t), \nonumber
\end{eqnarray}
so that
\begin{equation}\label{j2int}
\| J_2 \|_{L^p}  \le \frac{C}{R} \|f\|_{L^p}.
\end{equation}

Finally, integrating by parts twice and using that $1+z^2/R^2$ vanishes at $-iR$ and at $iR$, we obtain that
\begin{equation}\label{j3}
J_3(t) = \Bigl\| \frac{1}{t^2}
\int_{-R}^R \varphi''(s) e^{ist} \, ds
\Bigr\| \le  \frac{C_{R,f}}{t^2},
\end{equation}
where $\varphi(s) = (1-s^2/R^2)^n \widehat f(is)/s$.  Since $J_3$ is bounded on $(0,1)$, the inequalities \eqref{j1int}, \eqref{j2int} and \eqref{j3} imply that $g \in L^p(\mathbb R_+, X)$.

If $1<p<\infty$, then $g$ is H\"older continuous of exponent $1 - p^{-1}$ and in particular $g$ is uniformly continuous.  Since $g \in L^p(\R_+,X)$, it follows that $g \in C_0(\R_+,X)$.  If $p=1$, then $\lim_{t\to\infty} g(t)$ exists.  Since $g \in L^1(\R_+,X)$ it follows that the value of the limit is $0$.  In each case $g$ is a bounded function in $L^p(\R_+,X)$, so it is also in $L^r(\R_+,X)$ when $p \le r \le \infty$.
\end{proof}

When $p=\infty$, we cannot conclude that $g \in C_0(\R_+,X)$ in Theorem \ref{lpingham} (for example, let $f(t) = e^{it}$, so $\widehat f(\lambda) = (\lambda-i)^{-1}$ and $g(t) = ie^{it}$).  In this case one needs further assumptions to conclude that $g \in C_0(\R_+,X)$.  See Theorem \ref{inghamsing} and \cite{ArBa88} for results of this type.

\section{$L^p$-tauberian theorems with rates}\label{se4}

In this section we give a quantified version of Theorem \ref{lpingham}.  Under the assumptions of that result, the value of $R$ in the proof is confined to a small range.  If $\widehat f$ extends analytically across the whole of $i\R$, then $R$ can be chosen to be arbitrarily large and to depend on $t$.  Under additional assumptions on the domain $\Omega$, where $\widehat f$ extends analytically, and on the growth of $\widehat f$ in $\Omega$, we show that $g(t):=\widehat f(0)-\int_{0}^t f(s)\, ds$ belongs to a certain weighted $L^p$-space with a weight $w$ determined in terms of the shape of $\Omega$ and the bound for $\widehat f$ in $\Omega$.

When $p=\infty$, the appropriate quantified version of Theorem \ref{lpingham} is  Theorem \ref{inghamrates}.  The condition \eqref{resbound} involves the function $M$ to measure both the shape of the region $\Omega_M$ to which $\widehat f$ can be extended and as a bound for the resolvent.  However these two factors do not always play an equal role, and the shape has more effect when $M$ grows slowly.   Examination of the proof of Theorem \ref{inghamrates} in \cite{BaDu} or \cite[Theorem 4.4.6]{ABHN01} shows that \eqref{resbound} can be replaced by
\[
 \|\widehat f(\lambda)\|\le K (1+|\Im\lambda|) M_{\log}(|\Im \lambda |), \qquad \lambda \in \Omega_M.
 \]
Here $K$ can be any constant, and the constant $C$ in \eqref{Mlogest} may depend on $K$.  This condition is certainly satisfied if $\widehat f$ grows slower than linearly, so the rate in \eqref{Mlogest} is independent of the rate of growth so long as it is sublinear and the region is fixed.

Now assume that
\[
 \|\widehat f(\lambda)\|\le M_2(|\Im \lambda |), \qquad \lambda \in \Omega_{M_1},
 \]
where $M_1(s) = K_1(1+s)^\alpha$ for some $K_1$ and  $\alpha>0$, and $M_2$ is any increasing continuous function.  Then for any $\e>0$, there exists $C_\e>0$ such that
\[
 \|\widehat f(\lambda)\|\le C_\e (M_2(1+|\Im \lambda |))^\e (1+|\Im\lambda|)^\alpha, \qquad \lambda \in \Omega_{C_\e M_1}.
 \]
This is proved in \cite[Lemma 3.4]{BoTo10} in the case when $M_2(s) =K_2(1+s)^\beta$ and the general proof is almost the same.  Instead of applying Theorem \ref{inghamrates} with $M(s) = \max(M_1(s), M_2(s))$ one can apply it with
\[
M(s) = C_\e' (M_2(1+s))^\e (1+s)^\alpha,
\]
where $C_\e'$ grows polynomially as $\e\to0+$.  If $M_2(s) = K_2(1+s)^\beta$ where $\beta>\alpha$ this improves the estimate \eqref{Mlogest} for $g(t) := \widehat f(0) - \int_0^t f(s) \, ds$ from $\|g(t)\| = O(t^{-1/\beta})$ to $\|g(t)\| = o(t^{-\gamma})$ for each $\gamma<1/\alpha$.  If $M_2$ grows exponentially, it improves the estimate from $\|g(t)\| = O((\log t)^{-1})$ to  $\|g(t)\| = o((\log t)^{-1})$.

On the other hand, when $f(\lambda)$ is the resolvent $R(\lambda, A)$ of the generator $A$ of a bounded $C_0$-semigroup, it is rather straightforward that it suffices that \eqref{resbound} holds for $\lambda \in i\R$. Indeed, if $\|R(is, A)\| \le M(|s|)$ for $s \in \mathbb R$, then the Neumann series expansion shows that, for any $a > 1$, the region $\Omega_{aM}$ is contained in the resolvent set of $A$  and $\|R(\lambda, A)\|\le \frac{a}{a-1} M(|\Im \lambda|)$ for $\lambda \in \Omega_{aM}$.  So in this case it is natural to describe the region and the resolvent growth by the same function $M$, up to constant multiples.  We shall return to this situation in Section \ref{sect6}.

In general, in Theorems \ref{x1} and \ref{lpsing} below, one can assume that the region $\Omega_M$ and the growth of $\widehat f$ in $\Omega_M$ at infinity, or near a singularity, are determined by two different functions $M_1$ and $M_2$ respectively (taking into account the interplay between these two mentioned above). The approach in this paper then estimates the rate of decay by a weight $w$ which is constrained by $M_1$ and $M_2$.  As indicated by Theorems \ref{borto} and \ref{example}, given that $M_1=M$ a natural weight is $w = w_{M,\log}$.  In Theorems \ref{x1} and \ref{lpsing} below, we give the most general form of $M_2$ known to us that allows the conclusion for the weight $w_{M,\log}$.  In other situations where $M_2$ has much faster or much slower growth our methods could produce a different weight $w$. However our results would become more technical and less transparent and, in this paper, we do not see any gain in such generality.

Theorem \ref{x1} below extends Theorem \ref{inghamrates} to include cases when $1 \le p < \infty$ and by allowing a slightly more general form of \eqref{resbound} in which $M_2$ is allowed to grow as fast as a polynomial function of $M_1$ and $s$.  In Theorem \ref{lpsing} we give a quantified version of Theorem \ref{inghamsing}, which is new even for $p=\infty$, to the best of our knowledge.  In the next section we shall show that Theorem \ref{x1} is optimal in a strong sense which will be made precise.

To prove Theorem \ref{x1}, we will need a few very basic facts on Carleson measures, which can be found in \cite{Garnett}, for example.  Recall that, if $f \in L^p (\mathbb R)$,  $P_y*f$ is the  Poisson integral of $f$, and $\sigma$ is a Carleson measure on the upper half-plane $\mathbb C^+$, then by the Carleson embedding theorem (see \cite[Theorem I.5.6]{Garnett}),
\begin{equation} \label{Cembed}
\int_{\mathbb C^+} |(P_y*f) (z)|^p \, d\sigma(z) \le C_p \int_{\mathbb R} |f(s)|^p \, ds, \qquad z=x+iy.
\end{equation}
An example of such a measure is given by $\sigma (B)=l (B\cap \Gamma)$ for Borel subsets $B$ of $\C_+$, where $l$ is the Lebesgue length measure on a curve  $\Gamma=\{t + i \gamma(t): t \in [a,\infty)\}$,  satisfying
\[
l(\Gamma(t,\e)) \le C \e,  \qquad t\ge a, \e > 0,
\]
where $\Gamma(t,\e) = \{ s+i\gamma(s): |s-t|<\e\}$.  Such curves are often called {\it Carleson curves}, see \cite{BoKa97}. Note that $\Gamma$ is a Carleson curve if  $\gamma \in C^1([a,\infty))$ and $|\gamma'(t)|\le C, t \in [a,\infty)$, for some $C>0$; or if $\gamma$ is bounded and monotonic.

The proof will use some of the ingredients of Theorem \ref{lpingham}, with the parameter $R := R(t)$ now being chosen to depend on $t$.  The Carleson embedding \eqref{Cembed} will be combined with \eqref{j1est} and \eqref{j2est} to estimate $L^p$-norms of $J_1$ and $J_2$.  However, in order to use \eqref{j3} the growth of $R(t)$  has to be slow, and then the estimates for $J_1$ and $J_2$ become inefficient.  Consequently the proof requires a more delicate choice of contour, and this follows the same lines as \cite{BaDu} (see also \cite[Theorem 4.4.14]{ABHN01}).

\begin{theorem}\label{x1} Let $X$ be a Banach space, and
let  $f \in L^p(\mathbb R_+, X)$ where $1\le p\le\infty$. Let $M : \R_+ \to [2,\infty)$ be a continuous, increasing function such that $\widehat f$ admits an analytic extension to $\Omega_M$ and
\begin{equation}\label{omega}
\|\widehat f (z)\| \le K (1 + |\Im z|)^\alpha M(|\Im z|)^{\beta}, \qquad z\in \Omega_M,
\end{equation}
for some non-negative $K$, $\alpha$ and $\beta$.  Let
$$
g(t)=\widehat f(0) -\int_{0}^{t} f(s) \, ds.
$$
Then
\begin{equation*}
g \in L^p(\R_+,w,X) \cap C_0(\R_+,X),
\end{equation*}
where $w(t) = \wml(kt)$ for some $k>0$ depending only on $\alpha,\beta$ and $p$.  Moreover there is a number $C$ depending only on $p,M,K,\alpha,\beta$, such that
\begin{equation} \label{lprate}
\|g\|_{L^p(\R_+,w,X)} \le C\|f\|_p.
\end{equation}
\end{theorem}

\begin{proof} We shall present the proof for the case when $p < \infty$. The changes required for $p=\infty$ are straightforward, and most of the proof in that case is the same as in \cite{BaDu} and \cite[Theorem 4.4.6]{ABHN01}.

Without loss of generality we assume that $\widehat f$ is continuous up to $\partial\Omega$.  This can be arranged by replacing $M$ by $aM$, where $a>1$.
For $t > 0$ we shall use  (\ref{Ciform}) for some $n\in\N$ to be chosen later,  and
$$
R=R(t)= \wml(kt)
$$
for some (small) positive number $k$ to be chosen later.

Let $\gamma_1=R\mathbb T\cap \mathbb C_+$, $\gamma_2 =R\mathbb T \cap
\mathbb C_-$, $\gamma_3=(\pm iR+\mathbb R)\cap \Omega \cap \mathbb C_-$,  $\gamma_4= \{z \in \partial\Omega: |\operatorname{Im} z| < R\}$.  We shall apply \eqref{Ciform}, and also \eqref{sumj} in an adjusted form where
the interval $I = [-iR,iR]$ is replaced by $\gamma_3 \cup \gamma_4$.  Then
$$
2\pi \|g(t)\| \le J_1(t)+J_2(t)+I_3(t)+I_4(t),
$$
where $J_1(t)$ and $J_2(t)$ are as in the proof of Theorem \ref{lpingham}, and
$$
I_j(t) = \Bigl\|\int_{\gamma_j} \Bigl(1+\frac
{z^2}{R^2}\Bigr)^n \widehat f(z)
e^{zt}\,\frac{dz}{z}\Bigr\|, \qquad j=3,4.
$$

Now, by \eqref{j1est} and \eqref{Cembed},
\begin{multline} \label{i1est}
\int_0^{\infty} J_1(t) ^p \wml(kt)^p \,dt
 \le c \int_0^\infty \frac{(P_{1/R(t)}*h)(t)^p}{R(t)^p} \wml(kt)^p \,dt
 \\
= c\int_0^\infty (P_{1/R(t)}*h)(t)^p \, dt
\le c\int_\R h(t)^p \, dt \; = \; c \|f\|_{L^p}^p,
\end{multline}
where
$h=\|f\|_X$ on $\mathbb R_+$, $h=0$ on $\mathbb R_-$.  Here we have used \eqref{Cembed} which depends on $\Gamma:= \{t+i/R(t): t \ge 0\}$ being a Carleson curve, and that follows from monotonicity of $R$.
Similarly, \eqref{j2est} leads to
\begin{equation} \label{i2est}
\int_0^{\infty} J_2(t)^p \wml(kt)^p \,dt  \le  C \|f\|_{L^p}^p.
\end{equation}
Furthermore,
\begin{multline}  \label{i3est}
I_3(t) \\= \Bigl\| \int_{0}^{1/M(R(t))} \Bigl( 1 + \frac{(-s \pm iR(t))^2}{R(t)^2} \Bigr)^n \frac {\widehat f(-s\pm iR(t)) e^{(-s \pm iR(t))t}} {-s \pm iR(t)} \, ds \Bigr\|  \\
 \le C\int_{0}^{1/M(R(t))} \frac{R(t)^\alpha M(R(t))^{\beta} s^n}{R(t)^{n+1}} \,ds
  \le \frac{C}{R(t)^{n+1-\alpha}M(R(t))^{n+1-\beta}}.
\end{multline}
If $n>\alpha$ and $n> \beta - 1 + 1/p$ we obtain from \eqref{rmint4} that
\begin{multline} \label{i3est2}
\int_0^{\infty} I_3(t) ^p \,\wml(kt)^p \, dt  \\ \le C \int_0^{\infty} \frac{dt}{M(R(t))^{(n+1-\beta)p} R(t)^{(n-\alpha)p}}   < \infty.
\end{multline}

Finally, for $z \in \gamma_4$, $z = -1/M(|s|) + is$ where $|s| \le  R(t)$, $|\widehat f(z)| \le K R(t)^\alpha M(R(t))^\beta$, and
\[
\left| e^{zt} \right| = \exp \left( - t/M(|s|) \right) \le \exp \left( -t/M(R(t)) \right). \\
\]
Moreover, $|1+z^2/R(t)^2|$ and $|z^{-1}|$ are bounded independently of $t$, and the length of $\gamma_4$ is at most $2+ 2R(t) \le 4 R(t)$.  Hence
\begin{equation} \label{i4est1}
I_4(t) \le 4K R(t)^{\alpha+1} M(R(t))^{\beta} \exp \left( -t/M(R(t)) \right).
\end{equation}
For  $t\ge k^{-1} M_{\log}(1)$, we have $M_{\log}(R(t)) = kt$ and
\begin{eqnarray*}
I_4(t) &\le& 4K R(t)^{\alpha+1} M(R(t))^{\beta} \exp \left(- k^{-1} \log(R(t)M(R(t))) \right) \\
&=& 4K R(t)^{-k^{-1}+\alpha+1} M(R(t))^{-k^{-1}+\beta}.
\end{eqnarray*}
It follows from \eqref{rmint4} that
\begin{multline} \label{i4est}
\int_{0}^{\infty} I_4(t) ^p \, \wml(kt)^{p}\,dt \\ \le c + c \int_{0} ^\infty \frac {R(t)^p}{R(t)^{p(k^{-1}-\alpha-1)}M(R(t))^{p(k^{-1}-\beta)}} \, dt < \infty,
\end{multline}
if $k < (\alpha+2)^{-1}$ and $k < (\beta+1)^{-1}$.  Given any such $k$, \eqref{i1est}, \eqref{i2est}, \eqref{i3est2} and \eqref{i4est} show that
\begin{multline*}
\int_0^\infty \|g(t)\|^p \wml(kt)^{p}
\,dt\\
 \le c \int_0^{\infty} (J_1(t)+J_2(t)+I_3(t)+I_4(t))^p \wml(kt)^{p}
\,dt \le C < \infty,
\end{multline*}
for some
$C$ depending only on $p,M,K,k,\alpha,\beta$ and $\| f \|_{L^p}$.  It follows that $g \in C_0(\R_+,X)$ as in the proof of Theorem \ref{lpingham}.
\end{proof}

Next we want to give a quantified version of Theorem \ref{inghamsing} in the $L^p$-setting.  To the best our knowledge, Theorem \ref{lpsing} is new even in the case when $p=\infty$.  Martinez \cite{Ma11} has given a quantified result in the $L^\infty$-setting, but she estimates a rate of convergence in a limit which is different from \eqref{singlim}.

Let $\eta \in \R \setminus \{0\}$ and $M: [0,\infty) \to [2,\infty)$ be a continuous, increasing function.  We assume, without loss of generality, that $M$ is constant on $[0,|\eta|+1]$. Define
\begin{eqnarray*}
\widetilde M_\eta(s) &=&  \max \left( M(|s|), M \big(|s-\eta|^{-1} \big) \right), \qquad s\in\R, s \ne \eta, \\
&=& \begin{cases} M(|s|), &\qquad |s-\eta|\ge 1, \\ M(|s-\eta|^{-1}), &\qquad 0<|s-\eta|\le1, \end{cases} \\
\widetilde \Omega_{M,\eta} &=& \Big\{z \in\C : \Im z \ne \eta, \Re z > - \frac {1}{\widetilde M_\eta(\Im z)} \Big\} \cup \{s+i\eta : s>0\}.
\end{eqnarray*}

\begin{theorem} \label{lpsing}
Let $X$ be a Banach space, and $f \in L^p(\mathbb R_+, X)$ where $1 \le p \le \infty$.
Let
\[
F(t) = \int_0^t e^{-i \eta s} f(s) \, ds.
\]
 Assume that $F \in L^p(\R_+,X)$ and that $\widehat f$ admits an analytic extension to $\widetilde\Omega_{M,\eta}$ satisfying the following for all $z\in \widetilde\Omega_{M,\eta}$, and for some $c$ and $\beta\ge0$:
\begin{equation} \label{omegasing}
\|\widehat f (z)\| \le
\begin{cases} c M(|\Im z|)^{\beta},  \qquad & |\Im z-\eta| \ge 1, \\ c M(|\Im z-\eta|^{-1})^{\beta}, \qquad & 0 < |\Im z-\eta| \le 1. \end{cases}
\end{equation}
Let
$$
g(t) =\widehat f(0) -\int_{0}^{t} f(s) \, ds.
$$
\begin{enumerate}[\rm(a)]
 \item \label{lpsinga} If $M(s) \ge \kappa s^\alpha$ for some $\alpha>1,\kappa>0$, then there exists $k>0$ such that $g \in L^p(\R_+,w,X)$
where $w(t) = \wml(kt)$.
\item \label{lpsingb}  If $M(s) = \max(c_1,c_2s)$ for some $c_1,c_2>0$ and $w(t) = t^\gamma$ for some $\gamma \in (0,1)$, then $g \in L^p(\R_+,w,X)$.
\end{enumerate}
\end{theorem}

\begin{proof}
The proof of this combines ideas from the proof of Theorem \ref{x1} above with ideas from \cite{ArBa88} (see \cite[Theorem 4.4.8]{ABHN01}).
In place of \eqref{Ciform}, we use a formula
\[
g(t)= \sum_{j=1}^7 I_j(t),
\]
where
\[
I_j(t) = \frac {\eta^{2n}}{2\pi i(\eta^2-\e^2)^n}\int_{\gamma_j}
\Bigl(1+\frac {z^2}{R^2}\Bigr)^n \Bigl(1 + \frac {\e^2}{(z-i\eta)^2} \Bigr)^n g_j(t) e^{zt}\,\frac{dz}{z}.
\]
The values of $n$ and $k$ will be the same as in Theorem \ref{x1}, depending only on $p$.   The positive parameters $R$ and $\e$ will be functions of $t$ with $\e = 1/R$.
In addition, $R(t)$ will be a continuous, increasing function of $t$, with
$R(t) \ge 2\max(|\eta|,|\eta|^{-1})$.  We shall specify $R$ later.

Next we shall specify the contours $\gamma_j$ and functions $g_j$, for $j=1,\dots,7$ in turn, and in each case we shall either show that $I_j \cdot R \in L^p(\R_+,X)$ or obtain pointwise estimates for $I_j(t)$.  We shall write $J_j(t) = \|I_j(t)\|$ and $f_t = f \chi_{(0,t)}$, so that $\widehat f_t$ is the entire function given by
\[
\widehat f_t(z) = \int_0^t e^{-zs} f(s) \, ds, \qquad z \in \C.
\]

\noindent  $I_1$:  We take $\gamma_1 = R\T \cap \C_+$, and $g_1 = \widehat f - \widehat f_t$.   Since
\begin{equation} \label{bdR}
\left|1 + \frac {\e^2}{(z-i\eta)^2} \right|  \le  2, \qquad |z|\ge R,
\end{equation}
$J_1(t)$ can be estimated in the same way as in (\ref{j1est}) and we find that $J_1 \cdot R \in L^p(\R_+)$.

\noindent  $I_2$:  We take $\gamma_2 = R\T \cap \C_-$, and $g_2 = \widehat f_t$.  Using \eqref{bdR}, one can estimate $J_2(t)$ as in \eqref{j2est} and obtain that $J_2 \cdot R \in L^p(\R_+)$.

\noindent  $I_3$:  We take $\gamma_3$ to consist of the two line-segments $(\pm iR + \R) \cap \widetilde \Omega_{M,\eta} \cap \C_-$,  and $g_3 = \widehat f$.  The inequality (\ref{bdR}) holds on $\gamma_3$, so one can estimate $J_3(t)$ as in \eqref{i3est}.

\noindent $I_4$:  We take $\gamma_4 = \left\{ z \in\partial \widetilde\Omega_{M,\eta} : |\Im z| < R, |\Im z - \eta| > \e \right\}$, and $g_4 = \widehat f$.  Here
\[
\left| 1 + \frac{\e^2}{(z-i\eta)^2} \right| \le 2,
\]
so $I_4(t)$ can be estimated in the same way as \eqref{i4est1}.

\noindent $I_5$:  We take $\gamma_5$ to consist of the two line-segments $(i\eta  \pm i\e + \R) \cap \widetilde \Omega_{M,\eta} \cap \C_-$,  and $g_5 = \widehat f$.  Here $z = -s + i(\eta \pm \e)$, where $0 < s <  1/M(R)$, and
\[
\left| 1 + \frac{\e^2}{(z-i\eta)^2} \right| = \frac{|1 \pm 2i\e/s|}{1+(\e/s)^2} \le \frac{2s}{\e} = 2sR.
\]
Since $M(s)\ge cs$ for large $s$, we obtain
\begin{align*}
J_5(t) &\le C \int_0^{1/M(R)} s^n R^n M(R)^\beta e^{-st} \,ds \\
&\le
C M(R)^\beta R^n \int_0^{\infty} s^n e^{-st} \, ds \\
&\le  \dfrac{C M(R)^\beta R^n}{t^{n+1}}.
 \end{align*}

\noindent $I_6$:  We take $\gamma_6 = \{z\in \C_+: |z-i\eta| = \e\}$, with $g_6 = \widehat f - \widehat f_t$.  Here $z = i\eta + \e e^{i\theta}$, where $-\pi/2 <\theta < \pi/2$, and
\begin{eqnarray*}
e^{zt} (\widehat f(z) - \widehat f_t(z)) &=& e^{zt} \int_t^\infty e^{-\e e^{i\theta} s} e^{-i\eta s} f(s) \, ds \\
&=&  - e^{i \eta t} F(t)   +  e^{zt} \int_t^\infty \e e^{i\theta} e^{- \e e^{i\theta} s} F(s) \, ds.
\end{eqnarray*}
Hence
\[
\big\| e^{zt} (\widehat f(z) - \widehat f_t(z)) \big\| \le \|F(t)\| + \e \int_t^\infty e^{- \e (s-t) \cos\theta} \|F(s)\| \, ds.
\]
Moreover, $|1 + \frac{z^2}{R^2}| \le 2$ and $|1 + \frac{\e^2}{(z-i\eta)^2}| = 2\cos\theta$ and $|z|\ge |\eta|/2$.  It follows that
\begin{eqnarray*}
J_6(t)
&\le&  C \e \|F(t)\| + C \int_{-\pi/2}^{\pi/2} \cos\theta \left(\e \int_0^\infty e^{- \e s \cos\theta } \|F(s+t)\| \, ds \right) \e \, d\theta \\
&\le& C \e \|F(t)\| + C \e^2 \int_0^\infty \frac{\|F(s+t)\|}{1 + \e^2 s^2} \, ds \\
&\le& C \e \|F(t)\| + C \e (P_{1/\e} * H_1)(t),
\end{eqnarray*}
where $H_1 = \|F\|$ on $\R_+$, $H_1 = 0$ on $\R_-$.  Since $H_1 \in L^p(\R)$, it follows that $J_6 / \e = J_6 \cdot R \in L^p(\R_+)$.

\noindent $I_7$:  We take $\gamma_7 =  \{z\in \C_-: |z-i\eta| = \e\}$, with $g_7 = \widehat f_t$.  The estimates are similar to those for $I_6$.  Now we have $z = i\eta + \e e^{i\theta}$, where $\pi/2 < \theta < 3\pi/2$, and
\begin{eqnarray*}
e^{zt} \widehat f_t(z) &=& e^{i \eta t} F(t) + e^{zt} \int_0^t \e e^{i\theta} e^{- \e e^{i\theta} s} F(s) \, ds, \\
\left| e^{zt}  \widehat f_t(z) \right| &\le& \|F(t)\| + \e \int_0^t e^{-  \e (t-s) |\cos\theta|} \|F(s)\| \, ds, \\
J_7(t) &\le&  C \e \|F(t)\| + C \e (P_{1/\e} *  H_1)(t).
\end{eqnarray*}
Again we obtain $J_7 \cdot R \in L^p(\R_+)$.

It follows from the estimates above that $g \in L^p(\R_+,R,X)$, provided that $R(t)$ is chosen in such a way that the following functions of $t$ all belong to  $L^p(a,\infty)$ for some $a\ge0$ and some $n\ge1$:
\begin{enumerate}
\item  $R(t)^{-(n+1)} M(R(t))^{-(n+1-\beta)}$,
\item  $\exp(-t / M(R(t)) R(t)^{2} M(R(t))^\beta$,
\item  $M(R(t))^\beta R(t)^{n+1}/t^{n+1}$.
\end{enumerate}
In particular, let $R(t) = \max\big(\wml(kt),2|\eta|,2|\eta|^{-1}\big)$.  Then the first two functions belong to $L^p(\R_+)$, for all sufficiently large $n\ge1$ and all sufficiently small $k>0$, as seen in the proof of Theorem \ref{x1}, specifically \eqref{i3est2} and \eqref{i4est}.

If $M(s) \ge \kappa s^\alpha$ for some $\alpha>1$ and $\kappa >0$, then $R(t)\le C t^\gamma$ for some $\gamma < 1$ and for large $t$.  Using also \eqref{rmint}, we may estimate the third function as follows:
$$
\frac {M(R(t))^\beta R(t)^{n+1}}{t^{n+1}} \le \frac{C}{t^{(n+1)(1-\gamma)-\beta}}
$$
 so the third function is in $L^p(2,\infty)$ if $n > (\beta+\gamma)(1-\gamma)^{-1}$.  This completes the proof of statement (\ref{lpsinga}).
\vanish{
 and $R(t) = M_1^{-1}(t^\alpha)$, where $0<\alpha< 1 - p^{-1}$ if $p < \infty$ and $\alpha=1$ if $p=\infty$.  Then the functions above become
\begin{enumerate}
\item $t^{-\alpha n}$  for some $n\ge1$,
\item  $\exp(- t^{1-\alpha}R(t)) t^{\alpha} R(t)$,
 \item $t^{-(1-\alpha)}$.
\end{enumerate}
The function $r \mapsto \exp(- t^{1-\alpha}r)rt^\alpha$ is decreasing for $r > t^{-(1-\alpha)}$.    Since $R$ is increasing, $R(t)>c$ for all sufficiently large $t$ so
 \[
\exp(- t^{1-\alpha}R(t)) t^{\alpha} R(t) \le c t^\alpha \exp(-ct^{1-\alpha}).
\]
This function is in $L^p$.  Thus $g \in L^(\R_+,W,X)$ with
\[
W(t) = M_1^{-1}(t^\alpha).
\]
Since $L^p(\R_+,W,X) \subset L^p(\R_+,X)$, it follows that $g \in C_0(\R_+,X)$ in the same way as in Theorem \ref{lpingham}.}

For the statement \eqref{lpsingb}, we put $R(t) = \max\big((1+t)^\gamma,2|\eta|,2|\eta|^{-1}\big)$ where
$\gamma \in (0,1)$.  In this case the three functions are easily estimated as follows, for $t$ sufficiently large:
\begin{enumerate}
\item  $R(t)^{-(n+1)} M(R(t))^{-(n+1-\beta)} \le C t^{-(2(n+1)-\beta)\gamma}$,
\item  $\exp(-t / M(R(t))) R(t)^{2} M(R(t))^\beta  \le C \exp(-t^{1-\gamma}) t^{(\beta+2)\gamma}$,
\item  $M(R(t))^\beta R(t)^{n+1}/t^{n+1} \le C t^{\gamma\beta - (n+1)(1-\gamma)}$.
\end{enumerate}
Each of these functions is in $L^p(2,\infty)$ if $n$ is large enough.  This completes the proof of statement (\ref{lpsingb}).
\end{proof}

\begin{remark}  In the formulation of Theorem \ref{lpsing}, the assumption \eqref{omegasing} appears to be formally less general than the corresponding assumption \eqref{omega} in Theorem \ref{x1}, as terms which are polynomial in $\Im z$ are not included. This corresponds to the case $\alpha=0$ in \eqref{omega}. However in both cases \eqref{lpsinga} and \eqref{lpsingb} of Theorem \ref{lpsing} we consider functions $M$ which grow at least linearly.  Then such terms can be absorbed by changing the value of $\beta$, so the more general case is covered by the theorem.
\end{remark}

When the assumptions of Theorem \ref{lpsing} hold with $M$ growing sublinearly, for example $M(s) = \max(2,s^\alpha)$ for some $\alpha \in (0,1)$, one might hope to achieve the conclusion of the theorem with a weight $w$ growing faster than $t$.  However the following example shows that the rate of decay given under the weaker assumptions of Theorem \ref{lpsing}\eqref{lpsingb} is close to optimal even when $M$ grows sublinearly, at least for $p=\infty$.  However, if  the shape of the domain $\Omega_M$ and the growth of $\widehat f$ in $\Omega_M$ near $i\eta$ are determined by different functions M, then one can obtain a faster decay of rate than in Theorem \ref{lpsing}; see \cite[Proposition 4.3]{BaDu}.

\begin{example}
Let $X=c_0(\N)$, $(\beta_n)_n \in X$ with $\Re\beta_n>0$, and $(T(t))_{t\ge0}$ be the $C_0$-semigroup of contractions on $X$ defined by
\[
T(t)x  =  (\exp(it-\beta_n t) \alpha_n)_n, \qquad x = (\alpha_n)_n \in X.
\]
Let $f(t) = T(t)x$ where $x = (\beta_n)_n \in X$.  Since the generator has spectrum $\{i-\beta_n : n\ge1\} \cup \{i\}$, $\widehat f$ extends analytically at every point of $i\R \setminus \{i\}$ and
\[
\big\| \widehat f(is) \big\| = \sup_{n\ge1} \left| \frac {\beta_n}{i-is-\beta_n} \right|.
\]
Moreover,
\[
\int_0^t e^{-is} f(s) \, ds = \left(1 - \exp (-\beta_n t) \right)_n.
\]
This is bounded in $X$.   On the other hand,
\[
g(t) = \widehat f(0) - \int_0^t f(s) \, ds =  \left(\frac{\beta_n}{\beta_n-i}  \exp(it-\beta_n t) \right)_n.
\]

Now consider the case when $\beta_n>0$.  We will consider the conditions of   Theorem \ref{lpsing} with $\eta=1$ and $p=\infty$.  For $\lambda = i - re^{i\theta}$ where $r>0$ and $-\pi/2 < \theta < \pi/2$, we have
\[
\big\| \widehat f(\lambda) \big\| = \sup_{n\ge1} \Big| \frac {\beta_n}{re^{i\theta}-\beta_n} \Big| = \sup_{n\ge1} |1 - \beta_n^{-1}re^{i\theta}|^{-1} \le |\csc \theta|.
\]
Take $M(s) = \max(2,s^{1/2})$.  When $\lambda \in \widetilde\Omega_{M,1} \cap \C_-$, with $0< |\Im\lambda - 1|\le 1/4$, let $s = |\Im\lambda - 1|^{-1}$.  Then $|\tan \theta| > M(s)/s = s^{-1/2}$.  Hence
$$
\big\| \widehat f(\lambda) \big\| \le  |\csc \theta| < (1+s)^{1/2} \le C M(|\Im\lambda-1|^{-1}).
$$
Since $\widehat f$ is bounded elsewhere in $\widetilde\Omega_{M,1} \cap \C_-$, the condition \eqref{omegasing} of Theorem \ref{lpsing} is satisfied for $\beta = 1$.    Theorem \ref{lpsing}\eqref{lpsingb} can be applied  to show that $\|g(t)\|= o(t^{-\gamma})$ as $t \to\infty$, for any $\gamma<1$.  Even though the conditions of the theorem are satisfied with $M$ growing sublinearly,  $\|g(t)\|$ does not decrease faster than $t^{-1}$.  In fact,
\[
\|g(t)\| = \sup_{n\ge1} \frac{\beta_n e^{-\beta_n t}} {|i - \beta_n|}  \le \frac{1}{et}
\]
and
\[
\|g(\beta_n^{-1})\| \ge \frac {\beta_n}{(1+ \beta_n^2)^{1/2}e}.
\]
So the optimal rate of decay is $\|g(t)\| = O(t^{-1})$, see also \cite[Proposition 4.3]{BaDu}.

By choosing $\beta_n$ to have arguments tending to $\pi/2$, one can construct variants of this example in which $\|\widehat f(is)\|$ increases arbitrarily rapidly as $s\to1$.
\end{example}

\section{Optimality of weights}\label{se5}

We begin this section by showing that Theorem \ref{inghamrates} cannot be essentially improved in the case of polynomial rates.  At the end of the section we state a result for logarithmic rates.

\begin{theorem}\label{example}
Given $\alpha>0$, $p\in [1,\infty)$, and a positive function $\gamma
\in C_0(\mathbb R_+)$, there exists a function $f\in L^p (\mathbb
R_+)$ such that
\begin{enumerate} [\rm(a)]
\item \label{51a} $ \widehat f$ admits an analytic extension to the region
$$
\Omega:=\{z\in\mathbb C: {\rm Re}\, z >  -1/(1+|\Im z|)^\alpha\},
$$
with
\begin{equation}  \label{fhatbd}
| \widehat f (z)| \le (1 +  |\Im z|)^{\alpha/2}, \qquad z\in \Omega,
\end{equation}
and
\item \label{51b}
$$
\int_0^\infty  \Big|\widehat f(0) - \int_{0}^{t} f(s) \, ds\Big|^p \, \Bigl(\frac{t}{\gamma(t)\log (t+2)}\Bigr)^{p/\alpha}
\,dt=\infty.
$$
\end{enumerate}
\end{theorem}

The proof of Theorem \ref{example} is based on the following
proposition.   For a complex measure $\mu$ on $\C \setminus \Omega$ and $t\ge0$, we define
\begin{align}
\label{Lmu}
\XL\mu(t) &= \int_{\mathbb C \setminus
\Omega} e^{t\zeta}\,d\mu(\zeta), \\
 \label{Gmu}
\mathcal G\mu(t,z) &=\int_{\mathbb C \setminus \Omega} e^{t
\zeta}\,\frac{d\mu(\zeta)}{z-\zeta} \,,  \qquad z \in \Omega, \\
\label{Nmu}
\XN\mu(t) &= \int_{\mathbb C \setminus
\Omega}
e^{t\zeta}\,\frac{d\mu(\zeta)}{\zeta} \,.
\end{align}
These functions are related as follows:
\begin{equation} \label{GNmu}
\mathcal G\mu(0,z) = \widehat{\XL\mu}(z),  \quad \XN\mu(t) =  \int_0^t \XL\mu(s) \,ds - \widehat{\XL\mu}(0).
\end{equation}
While for the proof of Theorem \ref{example} we will need an estimate only for $\mathcal G \mu(0,z),$ the proposition below
provides a bound for $\mathcal G \mu(t,z)$ for all $t \ge 0.$ This more general bound will be needed for the proofs
of Theorems \ref{xyz} and \ref{t7}.
\begin{proposition}\label{prop}
Given $\alpha>0$ and $\beta > \alpha/2$, there exist arbitrarily large integers $k$, complex measures $\mu=\mu(k)$ with compact support in $\mathbb C\setminus \Omega$, and points $w=w(k)\in \mathbb C\setminus \Omega$, such that
\begin{equation} \label{wk}
|w| \asymp  \left(\frac{k}{\log k} \right)^{1/\alpha}, \quad k\to\infty, \qquad \supp\mu\subset
\{z:|z-w|<1\},
\end{equation}
 and, for some constants $C,c>0$ (depending on $\alpha$ and $\beta$) and some absolute constant $\rho>0$, we have
\begin{align}
&\hskip-10pt |\mathcal{L}\mu(t)|\le
C\chi_{\{t:|t-k|<k/2\}}e^{-\rho(t-k)^2/k}+ e^{-\rho t}, \label{X3}\\
&\hskip-10pt  |\mathcal{G}\mu(t,z)| \le  C\chi_{\{t:t\le 2k\}}
\big(|\Im z|^{\beta}\chi_{\{z:|z-w|<2\}}+ 1\big)+e^{-\rho t}, \label{XQ4}\\
&\hskip-10pt |\mathcal{N}\mu(t)| \ge c\Bigl(\frac{\log
k}{k}\Bigr)^{1/\alpha}\quad \text{if \;$(t-k)^2<k$},\label{X5}
\end{align}
and
\begin{equation}
\hskip-15pt |\mathcal{N}\mu(t)| \le C\Bigl(\frac{\log
k}{k}\Bigr)^{1/\alpha}e^{-\rho(t-k)^2/k}\chi_{\{t:|t-k|<k/2\}}+ e^{-\rho
t},\label{X6}
\end{equation}
for all $t\ge0$ and $z \in \Omega$.
\end{proposition}

\begin{proof}  In the course of this proof, we shall use the symbol $\rho$ to denote an absolute constant, which is strictly positive and is chosen to make   certain inequalities hold.  In each case the inequalities are (trivially) also true for smaller positive values, so that in the end we can take the minimum of the various values of $\rho$.  We shall consider integers $k\ge2$, and at several stages we shall assume that $k$ is sufficiently large that certain inequalities hold.  These inequalities involve $\rho$, but each inequality will hold for all integers $k$ which are sufficiently large (depending on $\rho$).  Thus for the final value of $\rho$, we establish all the estimates for all sufficiently large $k$.

Fix $\gamma\in(0,\beta-\alpha/2)$,  and define
\begin{eqnarray}
A&:=&2k\log k, \nonumber \\
\tau&:=&A^{k-1}/\sqrt k, \nonumber \\
q&:=&e^{2\pi i/k}, \nonumber\\
w&:=&iH-1, \label{wH}
\end{eqnarray}
where $H$ satisfies $k=\gamma H^\alpha\log H$.   Note that
\begin{gather}\label{Has}
|w| \sim H = \left( \frac{k}{\gamma \log H} \right)^{1/\alpha} \sim \left(\alpha \frac{k}{\gamma \log k} \right)^{1/\alpha}, \qquad k \to \infty,  \\
  \label{z-w}
|z-w| \ge 1 - H^{-\alpha}, \qquad z \in \Omega.
\end{gather}
Indeed if $|z-w|<1$, then $\Im z> H-1$ so $\Re z > -H^{-\alpha}$.
We take $k$ sufficiently large that $|w|\ge3$ and $|z-w| \ge 1/2$ for all $z \in \Omega$.

Consider the complex measure
$$
\mu:=\tau\sum_{1 \le s \le k}q^s \Big(1+ \frac{q^s}{Aw}\Big) \delta_{w+q^s/A},
$$
where $\delta_\zeta$ is the unit mass at $\zeta$.  Then \eqref{wk} holds.  To establish the remaining properties, we use the following elementary facts.

\begin{lemma}\label{sub1} For $k\ge j \ge 1$ and $z \in \C$ with $z^k\ne1$,
$$
\sum_{1 \le s \le k} \frac{q^{js}}{z -q^s}= \frac{kz^{j-1}}{ z^k-1} \,.
$$
\end{lemma}

\begin{proof}
For some polynomial $P$ with $\operatorname{deg} P < k$
we have
$$
\sum_{1 \le s \le k}\frac{q^{js}}{z-q^s}= \frac {P(z)}{z^k-1}.
$$
Since $P(qz)=q^{j-1}P(z)$, we obtain that $P(z)=c z^{j-1}$.  Expanding for large real $z$ we get
$$
\frac{P(z)}{z^k-1} =  \sum_{m\ge 0}\sum_{1 \le s \le k} \frac{q^{(m+j)s}}{z^{m+1}}=
\frac{k}{z^{k-j+1}} + O\left(z^{-(2k-j+1)}\right),\qquad z\to\infty,
$$
and, hence, $P(z)=kz^{j-1}$.
\end{proof}

\begin{lemma}\label{sub2} For $n\ge 1$ and $|z|\le 1$,
$$
\Bigl|e^z-\sum_{j=0}^n\frac{z^j}{j!}\Bigr|\le
2\frac{|z|^{n+1}}{(n+1)!}\,.
$$
\end{lemma}

\begin{proof} This follows easily from the Taylor expansion of the exponential
function. \end{proof}

\begin{lemma}\label{sub3} There exist absolute constants $C,c, \rho>0$ such that, for all $k\ge 3$,
\begin{align}
e^{k-t}(t/k)^k \max(\sqrt{t/k},1)&\le C e^{-\rho(t-k)^2/\max(t,k)},\,\, &&t\ge 0,  \label{sub31} \\
e^{k-t}(t/k)^k&\ge c,\qquad && |t-k|^2\le 2k, \label{sub32}\\
e^{-t} t^k \max(\sqrt{t},\sqrt{k}) / k! &\le C e^{-\rho(t-k)^2/\max(t,k)},\,\, &&t\ge 0. \label{sub33}
\end{align}
\end{lemma}

\begin{proof}   The inequalities \eqref{sub31} and \eqref{sub33} are equivalent (up to a change of $C$), by Stirling's formula.  We shall establish \eqref{sub31} and \eqref{sub32}.

Let $t = (1+s)k$, where $s>-1$.  Consider first the case when $|s| < 1$, so $t < 2k$.  From the Taylor expansion of $\log(1+s)$, one sees that the following inequalities hold for some $\gamma,\rho>0$ (more precisely, $0<\rho \le 1 - \log 2$),
\begin{align*}
-s + \log (1+s) &\le  - \rho s^2, &&|s|<1,  \\
-s + \log (1+s) &\ge - \frac{\gamma}{k}, && s^2 \le \frac{2}{k}, \quad k\ge3.
\end{align*}
Multiplying by $k$ and exponentiating, one obtains
\begin{align*}
e^{k-t}(t/k)^k &\le  e^{-\rho(t-k)^2/k}, &&0\le t < 2k,\\
e^{k-t}(t/k)^k&\ge e^{-\gamma}, && |t-k|^2\le 2k.
\end{align*}
The inequalities \eqref{sub31} and \eqref{sub32} follow, with $C=\sqrt2$ and $c = e^{-\gamma}$.

Now consider the case when $s\ge1$, so $t\ge 2k$.  It is easily verified that, for sufficiently small $\rho>0$  (more precisely, $0 < \rho \le (6-7 \log2)/6$), and all $s\ge1$ and $k\ge3$,
$$
 \log(1+s)  \le  \frac{6}{7} (1-\rho) s \le \left( \frac{2k}{2k+1} \right) \left( 1 - \frac{\rho s}{1+s} \right) s.
$$
Hence
$$
-ks + \left(k + \tfrac12\right) \log(1+s) \le  - \frac{\rho s^2 k}{1+s}.
$$
Exponentiating gives \eqref{sub31} with $C=1$.
\end{proof}

\begin{remark} \label{rem3}
It is easy to see that
$$
\Big| \frac{(k-t)^2}{\max(k,t)} - \frac{(k-1-t)^2}{\max(k-1,t)} \Big| \le 2, \qquad k\ge2, t \ge 0.
$$
Hence the estimates \eqref{sub31} and \eqref{sub33} remain true if $k$ is replaced by $k-1$ or $k-2$ on one side of the inequality (with the value of $\rho$ unchanged, but $C$ may change).
\end{remark}

Now we continue the proof of Proposition \ref{prop}.
Since $|w|\ge3$ and $|z-w| > 1/2$ for all $z \in \Omega$, it follows easily that
\begin{equation}  \label{zw}
\frac{|z|}{|w|\,|z-w|}  \le 4, \qquad z \in \Omega.
\end{equation}

In order to establish \eqref{X3}, we use that
\begin{equation}
\label{dop2}
\XL\mu(t) =\int_{\mathbb C \setminus
\Omega} e^{t\zeta}\,d\mu(\zeta)  =
\tau\sum_{1 \le s \le k}q^s \left(1+ \frac{q^s}{Aw}\right)e^{t(w+q^s/A)}.
\end{equation}
Then
\begin{align*}
\XL\mu(t)
&= \tau e^{tw} \sum_{n\ge 0} \sum_{1 \le s \le k} \left( q^s+ \frac{q^{2s}}{Aw} \right) \left( \frac{q^st}{A} \right)^n\frac 1{n!} \nonumber \\
&=
A^{k-1} \sqrt{k} e^{tw} \sum_{m\ge 1}\Bigl[\frac{t^{km-1}}{A^{km-1}}\cdot \frac
1{(km-1)!}+ \frac{t^{km-2}}{A^{km-2}}\cdot \frac 1{(km-2)!}\cdot
\frac 1{Aw} \Bigr] \nonumber \\
&= \frac{\sqrt{k} t^{k-1} e^{tw}}{(k-1)!}\sum_{m\ge 1}
\Bigl(\frac{t^k}{A^k}\Bigr)^{m-1}\frac{(k-1)!}{(km-1)!} \Bigl(1+
\frac {km-1}{tw}\Bigr). \nonumber
\end{align*}
When  $0\le t\le A$ and $k\ge3$, using  Lemma \ref{sub3} \eqref{sub33} and Remark \ref{rem3} we obtain
\begin{eqnarray*}
|\XL\mu(t)|
&\le& C  \frac{\sqrt{k} t^{k-1}e^{-t}}{(k-1)!}\sum_{m\ge 1}
\frac{(k-1)!m}{(km-1)!} \left(1 + \frac{k}{t}\right)\\
&\le& C \left( \frac{e^{-t}(k-1)^{1/2}t^{k-1}}{(k-1)!} + \frac{e^{-t}(k-2)^{1/2}t^{k-2}}{(k-2)!} \right) \sum_{m\ge1} \frac{ (k-1)! m}{(km-1)!} \\
& \le& C e^{-\rho (k-t)^2/\max(t,k)}.
\end{eqnarray*}
In particular,
$$
|\XL\mu(t)| \le \begin{cases} C e^{-\rho t/2} \quad &\text{when $0 < t \le k/2$}, \\ C e^{-\rho (k-t)^2/2k} \quad &\text{when $k/2 \le t \le 3k/2$}, \\ Ce^{-\rho t/9} &\text{when $3k/2 \le t \le A$}. \end{cases}
$$
When $t\ge A$, we obtain from \eqref{dop2} that
$$
|\XL\mu(t)| \le 2\tau e^{-t}k e^{t/A}=2 \sqrt{k} A^{k-1} e^{-t(1-1/A)} \le 2t^k e^{-t(1-1/A)} < e^{-\rho t},
$$
when $\rho \in (0,1)$ and $k$ is sufficiently large.  Thus we establish \eqref{X3}  for some $\rho$ and all sufficiently large $k$.
\smallskip

Next, to prove \eqref{XQ4}, we consider
\begin{equation*}
\mathcal G\mu(t,z)=\int_{\mathbb C \setminus \Omega} e^{t
\zeta}\,\frac{d\mu(\zeta)}{z-\zeta}=\tau e^{tw}\sum_{1\le s\le k}
\frac{Aq^s+q^{2s}/w}{A(z-w)-q^s}e^{tq^s/A}.
\end{equation*}
Take $\rho \in (0,1/4)$, and take $k$ sufficiently large that $\rho + 1/A \le 1/4$.  If $t\ge A$, then
\begin{align}\label{24}
|\mathcal G\mu(t,z)|&\le C\frac{A^{k-1}}{\sqrt k}e^{-t+(t/A)}k \\&= C 2^{k-1} k^{k-1/2} (\log k)^{k-1}e^{-t+(t/A)} \nonumber \\
&\le C 2^{k-1} k^{k-1/2} (\log k)^{k-1} e^{-3A/4} e^{-\rho t} \nonumber \\
&= C 2^{k-1} (\log k)^{k-1} k^{-(k+1)/2} e^{-\rho t} \nonumber\\
&\le  e^{-\rho t}, \nonumber
\end{align}
for sufficiently large $k$.

Now, we consider the case when $0 \le t\le A$. By Lemma~\ref{sub2} (applied with $z=tq^s/A$) we have
\begin{multline*}
\Bigl|\sum_{1\le s\le k}\frac{Aq^s+q^{2s}/w}{A(z-w)-q^s}e^{tq^s/A}\Bigr|
\\ \le \Bigl|\sum_{1\le s\le k}\sum_{0\le j<k-1}
\frac{Aq^s+q^{2s}/w}{A(z-w)-q^s}\cdot\frac{t^jq^{sj}}{A^jj!}\Bigr|
+\frac{Ckt^{k-1}}{A^{k-1}(k-1)!}.
\end{multline*}
By Lemma~\ref{sub1},
\begin{multline*}
\lefteqn{\sum_{1\le s\le k}\frac{1}{A(z-w)-q^s}\sum_{0\le j<k-1}
\Bigl(\frac{t^jA^{1-j}q^{s(j+1)}}{j!}+
\frac{t^jA^{-j}q^{s(j+2)}}{j!w} \Bigr)} \\
=
\frac{k}{A^k(z-w)^k-1}\!\sum_{0\le j<k-1}
\Bigl(\frac{t^jA^{1-j}}{j!}A^j(z-w)^j+
\frac{t^jA^{-j}}{j!w}A^{j+1}(z-w)^{j+1}  \Bigr)  \\
=
\frac{k}{A^k(z-w)^k-1}\!\sum_{0\le j<k-1} \frac{t^jAz}{j!w}(z-w)^j.
\end{multline*}
From this we obtain
\begin{eqnarray}
\lefteqn{|\mathcal G\mu(t,z)|} \nonumber \\
  &\le& \frac{cA^{k-1}e^{-t}}{\sqrt{k}}  \Bigl(\frac{kt^{k-1}}{A^{k-1}(k-1)!}
+\frac{k}{A^k|z-w|^k}\sum_{0\le j<k-1}
\frac{t^jA|z|}{j!|w|}|z-w|^j\Bigr) \nonumber \\
&=& \frac{ce^{-t}t^{k-1} \sqrt{k}}{(k-1)!}
+\frac{ce^{-t}\sqrt{k}}{|z-w|^k}\sum_{0\le j<k-1}
\frac{t^j|z|}{j!|w|}|z-w|^j \nonumber \\
&=:& X + Y.  \nonumber
\end{eqnarray}
By  Lemma~\ref{sub3} \eqref{sub33} and Remark \ref{rem3},
\begin{equation} \label{Xest}
X \le Ce^{-\rho(t-k)^2/\max(t,k)} \le \begin{cases} C e^{-\rho t/4} \quad &\text{if $2k \le t \le A$}, \\
 C  \quad &\text{if $0 \le t \le 2k$}.\end{cases}
\end{equation}

Now we consider the term $Y$.   Assume first that $2k \le t \le A$.  By Lemma~\ref{sub3} \eqref{sub33} we have
\begin{align*}
Y &\le C \frac{\sqrt{k} |z|}{|w|} \sum_{0 \le j < k-1} |z-w|^{-(k-j)} e^{-\rho(t-j)^2/t} \\
&\le C  e^{-\rho t/4} \frac{\sqrt{k} |z|}{|w|} \sum_{0 \le j < k-1} |z-w|^{-(k-j)}.
\end{align*}
By \eqref{z-w}, we can assume that $k$ is sufficiently large that  $|z-w|> e^{-\rho/4}$ for all $z \in \Omega$.  Using \eqref{zw}, we have
$$
Y \le C e^{-\rho t/4} k^{3/2} \frac{|z|}{|w|\,|z-w|} e^{\rho k/4}  \le C e^{-\rho t/9}.
$$
 From this and \eqref{Xest}, we have
 \begin{equation}  \label{27a}
|\mathcal G\mu(t,z)|  \le Ce^{-\rho t/9} \quad \text{when $2k\le t \le A$}.
\end{equation}

Next assume that $|z-w|\ge 2$ and $ k/4 \le t \le 2k$.  By Lemma~\ref{sub3} \eqref{sub33},  we have $e^{-t} \sqrt{k} t^j/j! \le C$ for $0 \le j < k-1$.  Using also \eqref{zw}, we obtain
\begin{equation}\label{Yest1}
Y \le \frac{C |z|}{|w|} \sum_{0\le j < k-1} |z-w|^{-(k-j)} \le \frac{C |z|}{|w| |z-w|} \le 4C.
\end{equation}
Next, assume that $|z-w|\ge 2$ and $t< k/4$. Let $m$ be the integer part of $k/2$.  Since $e^{-t} t^j/j! \le 1$,
\begin{eqnarray*}
\frac{e^{-t}\sqrt{k}}{|z-w|^k}\sum_{0\le j \le m} \frac{t^j|z|}{j!|w|}|z-w|^j
&\le&
\frac{\sqrt{k}|z|}{|w|} \sum_{0\le j \le m} |z-w|^{-(k-j)} \\
&\le&
\frac{ 2\sqrt{k} |z|}{|w| \, |z-w|^{k/2}} \\
&\le&
\frac{ 4\sqrt{k}}{2^{k/2}} \frac{|z|}{|w| \, |z-w|} \\
&\le& c.
\end{eqnarray*}
For $j>m$, by Lemma \ref{sub3} \eqref{sub33},
$$
\frac{e^{-t}t^j}{j!}  \le \frac{Ce^{-\rho(j-t)^2/j}}{\sqrt{t}} \le \frac {C}{\sqrt{t}} e^{-\rho k /8}.
$$
Hence
$$
\frac{e^{-t}\sqrt{k}}{|z-w|^k}\sum_{m< j < k-1} \frac{t^j|z|}{j!|w|}|z-w|^j
\le
\frac{C \sqrt{k} e^{-\rho k /8}}{\sqrt{t}} \frac{|z|}{|w|\, |z-w|}  \le C.
$$
Thus we have
$$
Y \le C \quad \text{when $|z-w|\ge 2$ and $t< k/4$}.
$$
In combination with \eqref{Xest} and \eqref{Yest1}, this gives
\begin{equation}
|\mathcal G\mu(t,z)| \le C
 \quad \text{when $|z-w|\ge2$ and $0< t \le 2k$}.\label{2u2}
\end{equation}

Next assume that $|z-w|\le 2$ and $t\le 2k$.  Using \eqref{z-w} we have
\begin{align*}
Y &=  \frac{e^{-t}\sqrt{k}}{|z-w|^k}\sum_{0 \le j < k-1} \frac{t^j|z|}{j!|w|}|z-w|^j \\
&\le \frac{\sqrt{k} |z|}{|w|} \max\left(|z-w|^{-1}, |z-w|^{-k}\right)  \le C\sqrt{k} (1-H^{-\alpha})^{-k}.
\end{align*}
Take $\gamma' \in (\gamma, \beta - \alpha/2)$ and $\gamma'' \in (\gamma',\beta-\alpha/2)$.   By \eqref{Has},
$$
H^\alpha \ge \frac{\alpha k}{\gamma' \log k}
$$
for all sufficiently large $k$.  Moreover
$$
e^{-\gamma'' s} + \gamma' s \le 1
$$
for all sufficiently small $s>0$.  Putting $s = (\log k)/(\alpha k)$ for sufficiently large $k$ we have
$$
1- \frac{\gamma' \log k}{\alpha k} \ge k^{-\gamma''/(\alpha k)}.
$$
Thus, for sufficiently large $k$,
$$
Y  \le c\sqrt{k} \left(1- \frac{\gamma' \log k}{\alpha k} \right)^{-k} \le C k^{1/2 + \gamma''/\alpha}.
$$
Since $1/2 + \gamma''/\alpha < \beta/\alpha$, it follows from \eqref{Has} that
$$
|\Im z|^\beta \ge (\Im w -2)^\beta = (H-2)^\beta  > k^{1/2+\gamma''/\alpha}
$$
for sufficiently large $k$.  Thus
$$
Y \le c |\Im z|^\beta
$$
for sufficiently large $k$.  Using also \eqref{Xest},
\begin{equation}
|\mathcal G\mu(t,z)|  \le C |\Im z|^\beta \quad \text{when $|z-w| \le 2$ and $t \le 2k$}. \label{27}
\end{equation}
Together, \eqref{24}, \eqref{27a}, \eqref{2u2}, and \eqref{27}
prove \eqref{XQ4}.
\smallskip

Finally, for \eqref{X5} we consider
\begin{align} \label{dop3}
\XN\mu(t) &= \int_{\mathbb C \setminus \Omega}
e^{t\zeta}\,\frac{d\mu(\zeta)}{\zeta} \\&= \frac{\tau
e^{-t+iHt}}{w}\sum_{1 \le s \le k}q^se^{q^st/A}  \nonumber\\
 &= \frac{\tau e^{-t+iHt}}{w} \sum_{n\ge0}\sum_{1 \le s \le k} q^s \left( \frac{q^s t}{A} \right)^n\frac 1{n!} \nonumber \\
&= \frac{e^{iHt}}{w} A^{k-1} \sqrt{k} e^{-t} \sum_{m\ge1}\frac{t^{km-1}}{A^{km-1}}\, \frac 1{(km-1)!} \nonumber\\
&= \frac{e^{iHt}}{w} \frac{\sqrt{k} e^{-t} t^{k-1}}{(k-1)!} \sum_{m\ge1} \Bigl(\frac{t^k}{A^k}\Bigr)^{m-1}\frac{(k-1)!}{(km-1)!}. \nonumber
\end{align}
When $0 \le t \le A$, we obtain from Stirling's formula that
\begin{equation*}
c  e^{k-t}\frac{t^{k-1}}{k^{k-1}H} \le \frac{e^{-t}\sqrt{k}t^{k-1}}{|w|(k-1)!} \le |\XN\mu(t)| \le C \frac{e^{-t} \sqrt{k} t^{k-1}}{|w| (k-1)!} \le C e^{k-t}\frac{t^{k-1}}{k^{k-1}H} \,.
\label{X8}
\end{equation*}
By \eqref{Has}, Lemma~\ref{sub3} and Remark \ref{rem3},
\begin{equation}
|\XN\mu(t)| \Bigl(\frac {k}{\log k}\Bigr)^{1/\alpha}\ge c >0,\qquad
|t-k|^2\le k,
\label{X81}
\end{equation}
and
\begin{equation*}
|\XN\mu(t)| \le C \Bigl(\frac {\log k}{k}\Bigr)^{1/\alpha}e^{-\rho (t-k)^2/\max(t,k)},\qquad 0 \le t \le A.
\end{equation*}
Hence
\begin{equation} \label{X82}
|\XN\mu(t)| \le \begin{cases}  Ce^{-\rho t/2},\quad 0 \le t \le k/2,  \\[0.3 em]
C \Bigl(\dfrac {\log k}{k}\Bigr)^{1/\alpha} e^{-\rho (k-t)^2/2t}, \quad k/2 \le t \le 3k/2, \\[0.7 em]
Ce^{-\rho t/9}, \quad 3k/2 \le t \le A.
\end{cases}
\end{equation}
Finally, by \eqref{dop3} when $t \ge A$ we have
\begin{equation}
|\XN\mu(t)|\le \frac{C\tau e^{-t}}{H} ke^{t/A} \le \sqrt{k} t^{k-1} e^{-t(1-1/A)} \le  e^{-\rho t}
\label{X83}
\end{equation}
when $\rho\in(0,1)$ and $k$ is sufficiently large.  Now, \eqref{X81}--\eqref{X83} together yield \eqref{X5} and \eqref{X6}.
\end{proof}

\begin{proof}[Proof of Theorem \rm\ref{example}]
Without loss of generality, we can assume that $\gamma$ is non-increasing and $\gamma(t) \ge t^{-1}$ for all $t\ge1$.  Choose $\beta$ with $\alpha/2 < \beta \le \alpha/2 + \alpha/(2p)$.  Let $\{k_n\}\subset \mathbb N$ be a sequence of integers as in Proposition \ref{prop} such that
$$
k_1 \ge 3, \quad  k_n \ge 3k_{n-1} \; (n\ge2), \quad 2^n \gamma\big(k_n-\sqrt{k_n}\big)^{1/\alpha} \to 0.
$$
For each $n\ge1$, let $\mu_n$ be the corresponding complex measure with compact support in $\C \setminus \Omega$, let $f_n:= \mathcal{L}\mu_n$, and let $w_n$ be the corresponding point in $\C \setminus \Omega$ so that
\begin{equation*}
|w_n| \asymp  \left(\frac{k_n}{\log k_n} \right)^{1/\alpha}, \quad n\to\infty,
\qquad \supp \mu_n \subset \{z: |z-w_n|<1\}.
\end{equation*}
Then the following hold:
\begin{enumerate}[\rm(i)]
\item by \eqref{GNmu},
$$
\widehat f_n(z) = \mathcal G\mu_n(0,z),  \quad \XN\mu_n(t) = \int_0^t f_n(s) \,ds - \widehat f_n(0),
$$
\item by \eqref{X3},
\begin{multline} \label{T11}
\|f_n\|_{L^p(\mathbb R_+)} \\
\le C \Bigl( \int_{k_n/2}^{3k_n/2} e^{-p\rho (t-k_n)^2/k_n}\, dt \Bigr)^{1/p} + \Bigl(\int_0^\infty e^{-p\rho t} \, dt \Bigr)^{1/p} \\
\le C k_n^{1/(2p)} \Bigl(\int_\R  e^{-\rho p u^2} \,du \Bigr)^{1/p} + (p \rho)^{-1/p}
\le C k_n^{1/(2p)},
\end{multline}
\item by \eqref{XQ4}, for $z\in\Omega$,
\begin{equation} \label{T12}
|\widehat f_n(z)| \le C\big(|\Im z|^\beta \chi_{\{z:|z-w_n|\le 2\}} + 1\big)
\le C(1+ |\Im z|)^{\alpha/2} k_n^{1/(2p)},
\end{equation}
where we have used that  $|\Im z| \le |w_n|+2 \le Ck_n^{1/\alpha}$ if $|z-w_n|\le2$, by \eqref{wk},
\item by \eqref{X5}, when $(t-k_n)^2 < k_n$,
\begin{equation*}
\Big|\widehat f_n(0) - \int_0^t f_n(s) \, ds\Big|\ge c_1\Bigl(\frac{\log
k_n}{k_n}\Bigr)^{1/\alpha},
\end{equation*}
\item by \eqref{X6},
\begin{multline*}
\Big|\widehat f_n(0)- \int_0^t f_n(s) \, ds \Big|\\ \le C\Bigl(\frac{\log k_n}{k_n}\Bigr)^{1/\alpha}e^{-\rho(t-k_n)^2/k_n}\chi_{\{t:|t-k_n|<k_n/2\}}+ c_2 e^{-\rho t}.
\end{multline*}
\end{enumerate}
Let
$$
f=\sum_{n=1}^\infty 2^{-n}k_{n}^{-1/(2p)}f_{n}.
$$
By \eqref{T11}, the series converges in $L^p(\mathbb R_+)$.  Property (\ref{51a}) of Theorem \ref{example} follows from \eqref{T12}.

For $k_n - \sqrt{k_n} \le t \le k_{n} + \sqrt{k_n}$, we have $|t-k_m| > k_m/2$ for $m \ne n$ and hence
\begin{align*}
\left|\widehat f(0) - \int_0^t f(s) \, ds  \right|
 &\ge 2^{-n}k_{n}^{-1/(2p)} \left| \XN\mu_n(t) \right| - \sum_{m \ne n} 2^{-m}k_{m}^{-1/(2p)} \left| \XN\mu_m(t) \right| \\
&\ge  c_1  2^{-n}k_{n}^{-1/(2p)} \left( \frac{ \log k_{n}}{k_{n}} \right)^{1/\alpha} - c_2 e^{-\rho t}.
\end{align*}
Hence
\begin{align*}
\lefteqn{\hskip-10pt\int_{k_n-\sqrt{k_n}}^{k_{n}+\sqrt{k_n}} \left( \left|\widehat f(0) - \int_0^t f(s) \, ds\right| + c_2 e^{-\rho t} \right)^p \Bigl(\frac{t}{\gamma(t)\log (t+2)}\Bigr)^{p/\alpha}
\,dt}  \\
&\ge 2 \sqrt{k_n} c_1^p 2^{-pn} k_{n}^{-1/2} \left( \frac{ \log k_{n}}{k_{n}} \right)^{p/\alpha} \left( \frac{k_n }{2 \gamma(k_n-\sqrt{k_n}) \log(2k_{n} + 2)} \right)^{p/\alpha}\\
&\to\infty.
\end{align*}
Since $\gamma(t) \ge t^{-1}$, we have that $t \mapsto e^{-\rho t} \Bigl(\frac{t}{\gamma(t)\log (t+2)}\Bigr)^{1/\alpha}$ is in $L^p(\R_+)$, and property (\ref{51b}) follows.
\end{proof}

\begin{remark}  Given $\gamma > \alpha/2 - \alpha/(2p)$, one can find $f \in L^p(\R_+)$ satisfying the properties of Theorem \ref{example} with \eqref{fhatbd} replaced by
$$|\widehat f(z)| \le (1 + |\Im z|)^\gamma, \qquad z \in\Omega.
$$
This is achieved by choosing $\beta = \gamma + \alpha/(2p)$ in the proof above.
\end{remark}

In the rest of this section we briefly discuss optimality of Theorem \ref{lpingham} for logarithmic rates.  When $M(s) = (\log (2+s))^\alpha$ and the assumptions of Theorem \ref{x1} are satisfied, \eqref{lprate} shows that
$$
\int_0^\infty
\Big|\widehat f(0)
-\int_{0}^{t} f(s) \, ds \Big|^p \, e^{\gamma t^{1/(\alpha+1)}}\,dt < \infty
$$
for some $\gamma>0$.  The following analogue of Theorem~\ref{example} shows that this may not hold for all $\gamma$, and it follows that Theorem~\ref{lpingham} is optimal in this case, up to possible changes of $k$.

\begin{theorem}\label{exampleB}
Given $\alpha>0$ and $p\ge 1$, there exist $\gamma>0$ and a function $f\in L^p (\mathbb R_+)$ such that
\begin{enumerate} [\rm(a)]
\item $ \widehat f$ admits a bounded analytic extension to the region
$$
\Omega:= \big\{z\in\mathbb C: \Re z >  -1/(\log (2+|\Im z|))^\alpha \big\},
$$
and
\item
$$
\int_0^\infty
\Big|\widehat f(0)
-\int_{0}^{t} f(s) \, ds \Big|^p \, e^{\gamma t^{1/(\alpha+1)}}\,dt=\infty.
$$
\end{enumerate}
\end{theorem}

The proof is based on a modification of Proposition~\ref{prop}:

\begin{proposition}\label{propA}
There exist arbitrarily large integers $k$, complex measures $\mu=\mu(k)$ with compact support in $\mathbb C\setminus \Omega$, and points $w=w(k)\in \mathbb C\setminus \Omega$, such that
\begin{equation*} \label{wkA}
\frac{\log|w|}{k^{1/(\alpha+1)}} =1+o(1), \quad k\to\infty, \qquad \supp\mu\subset
\{z:|z-w|<1\},
\end{equation*}
and, for some numbers $C,c>0$ depending on $\alpha$, for some absolute constant $\rho>0$, and for all $t\ge0$ and $z \in \Omega$, we have
\begin{enumerate}[\rm(i)]
\item $\displaystyle  |\mathcal{L}\mu(t)|\le C e^{-\rho t^{1/(\alpha+1)}}$,
\vskip5pt
\item $\displaystyle |\mathcal{G}\mu(t,z)|\le
C\chi_{\{t:t\le 2k\}}+e^{-\rho t}$,
{\vskip5pt}
\item $\displaystyle |\mathcal{N}\mu(t)| \ge c e^{-4 k^{1/(\alpha+1)}}$
if $(t-k)^2<k$,  and
\vskip5pt
\item $\displaystyle |\mathcal{N}\mu(t)| \le C e^{-\rho t}$ if $|t-k|>k/2$.
\end{enumerate}
\end{proposition}

The proof is similar to that of Proposition~\ref{prop}. Instead of \eqref{wH} we set $w=iH-1-2(\log H)^{-\alpha}$ where $H=\exp \left(k^{1/(\alpha+1)}\right)$.

\section{$L^p$-rates for semigroup orbits} \label{sect6}

Let $(T(t))_{t \ge 0}$ be a $C_0$-semigroup on a Banach space $X$, with generator $A$.  In this section, we apply our function-theoretic results to the study of the $L^p$-rates of decay for differentiable orbits of $(T(t))_{t \ge 0}$ (in other words, classical solutions of the abstract Cauchy problem \eqref{Cauchypr}).

We start with the following simple observations showing certain limitations of such studies. Let $x \in X$ be such that
\begin{equation}\label{orbitlp} T(\cdot)x \in L^p(\mathbb R_+, X),
\end{equation}
for some $p \in [1,\infty)$.  Since
$$
\|T(t)x\| = \left(\int_{t-1}^t \|T(t)x\|^p \, ds\right)^{1/p} \le K_1 \left(\int_{t-1}^t \|T(s)x\|^p \, ds\right)^{1/p}, \,\, t\ge1,
$$
where $K_1 = \sup_{0\le s\le1} \|T(s)\|$, it is immediate that $T(\cdot)x \in C_0(\R_+,X)$.  It then follows that $T(\cdot) R(\omega,A) x \in L^p(\R_+,X) \cap C_0(\R_+,X)$, for any $\omega  \in \rho(A)$.

Now assume that  $(T(t))_{t \ge 0}$ is bounded and  $R(\lambda, A)$ extends analytically to $i\mathbb R$ and
\begin{equation} \label{resbd}
\|R(is,A)\|\le M(|\Im s|), \qquad s \in \mathbb R,
\end{equation}
for a continuous increasing function $M$.  Then $R(\lambda, A)$ extends analytically to $\Omega_{2M}:=\{\lambda \in \mathbb C: \Re \lambda > - (2M(|{\Im} \lambda|))^{-1}\}$ and $\|R(\lambda, A)\| \le 2M(|{\Im}  \lambda |), \lambda \in \Omega_{2M}$.  By Theorem \ref{semigrouprates}, one has
\begin{equation}\label{supbound1}
\sup_{t \ge 0} \wml(ct) \|T(t)A^{-1}\|=:N <\infty,
 \end{equation}
for some $c >0$. Moreover, Theorem \ref{x1} yields the following $L^p$-analogue of  \eqref{supbound1}:
\begin{equation}\label{hope}
\int_{0}^{\infty}w_{M,\log}(kt)^p
\|T(t)A^{-1}x\|^p\,dt <\infty,
\end{equation}
for some $k>0$, if \eqref{orbitlp} holds.
However, \eqref{hope} is a trivial consequence of \eqref{supbound1} and the semigroup property, for any $k \in (0,c)$, in view of
\begin{equation*}
\int_{0}^{\infty} w_{M,\log}(kt)^p
\|T(t)A^{-1}x\|^p\,dt \le N^p   \int_{0}^{\infty} \big\|T\big((1-k/c)t\big)x \big\|^p \, dt.
\end{equation*}
Similarly, for bounded $C_0$-semigroups on Hilbert spaces, the optimal polynomial $L^{\infty}$-rates as in Theorem \ref{borto} give rise to the corresponding $L^p$-rates.

Thus, $L^{\infty}$-rates for smooth orbits of bounded operator semigroups yield  $L^p$-rates for such orbits in a straightforward manner.

However, if we are interested in the asymptotic behaviour of the families $(T_1 T(t) T_2)_{t \ge 0}$ arising in the study of decay of local energy as explained in the Introduction, then the semigroup property is violated and the problem of describing $L^p$-rates for $(T_1 T(t) T_2)_{t \ge 0}$ becomes distinct from its $L^{\infty}$-analogue studied in \cite{BaDu}, \cite{BoTo10}, \cite{Bu98}, \cite{Ch09}, \cite{ChSchVaWu13}, etc.  In the $L^p$-context, we can formulate the following corollary of Theorems \ref{lpingham} and \ref{x1}.   Since we are considering individual vectors $x$ we could suppress the operator $T_2$, but we retain it because it is standard practice in the study of decay of local energy to consider rates which are uniform in some respects.

Let $\omega_0 = \omega_0(T)$ be the exponential growth bound of a $C_0$-semigroup $(T(t))_{t \ge 0}$.  The following result is of interest when $\omega_0\ge0$, so that $R(\lambda,A)$ is defined for $\Re\lambda>\omega_0$.  In particular, it complements \cite[Th\'eor\`eme 3]{Bu98} which has applications to decay of local energy in exterior domains of odd dimension.

\begin{theorem}\label{energydecayabs}
Let $(T(t))_{t \ge 0}$ be a $C_0$-semigroup on a Banach space $X$, with generator $A$,  let $T_1$ and $T_2$ be bounded operators on $X$, and let $\omega >\max(\omega_0,0)$ be fixed.
If $x \in X$ is such that
\begin{enumerate}[\rm(i)]
\item $T_1T(\cdot)T_2x \in L^p (\mathbb R_+, X)$ for some $p \ge 1$, and
\item $T_1R(\cdot, A)T_2x$ extends to an analytic function $G$ on an open set $\Omega$ containing $\overline{\C}_+$,
\end{enumerate}
 then
 \begin{equation}\label{cutofflp}
 T_1T(\cdot)R(\omega,A)T_2 x \in L^p (\mathbb R_+, X) \cap C_0(\mathbb R_+, X).
 \end{equation}
If $M:\mathbb R_+\to [2,\infty)$ is continuous and increasing, $\Omega = \Omega_M$ as in \eqref{omm}, and for some $C,\alpha,\beta\ge 0$ we have
\begin{equation}\label{q91}
\|G(\lambda)\| \le C(1+|\Im \lambda|)^\alpha M(|{\Im}\lambda|)^\beta, \qquad \lambda \in \Omega_M,
\end{equation}
then
 \begin{equation}\label{cutofflprates}
 T_1T(\cdot)R(\omega,A) T_2 x \in L^p (\mathbb R_+, w, X),
 \end{equation}
 where $w(t) =\wml(kt)$
 for some $k>0$.
\end{theorem}

\begin{proof}
Let us start with the proof of \eqref{cutofflprates}.
Define
\begin{equation*}
F(t)=T_1T(t)(\omega R(\omega, A) - I)T_2x=T_1T(t)A R(\omega, A)T_2x, \qquad t \ge 0.
\end{equation*}
Observe that by the generalized Minkowski inequality,
\begin{eqnarray*}
\|T_1 T(\cdot) R(\omega, A)T_2 x\|_{L^p} &=& \left \| \int_{0}^{\infty}e^{-\omega s}T_1T(s+\cdot)T_2x \, ds \right \|_{L^p}\\
&\le& \int_{0}^{\infty}e^{-\omega s}\| T_1T(s+\cdot)T_2x\|_{L^p} \, ds\\
&\le& \omega^{-1} \|T_1T(\cdot)T_2x\|_{L^p}.
\end{eqnarray*}
Therefore, $F \in L^p(\mathbb R_+, X)$.
Moreover,
\begin{eqnarray*}
\widehat F(\lambda)&=&T_1(\omega R(\lambda, A)R(\omega, A)-R(\lambda, A))T_2x \\
&=&\frac{\lambda}{\omega-\lambda}T_1R(\lambda, A)T_2x - \frac{\omega}{\omega-\lambda} T_1R(\omega, A)T_2x \\
&=&\frac{\lambda}{\omega-\lambda}G(\lambda)-\frac{\omega}{\omega-\lambda}T_1R(\omega, A)T_2x, \qquad \Re\lambda > \omega_0.
\end{eqnarray*}
Hence, by assumption, $\widehat F$ extends analytically to $\Omega_M$ and there exists $c >0$
such that $\|\widehat F(\lambda)\| \le c M(|\Im \lambda|), \lambda \in \Omega_M$.
Now, applying Theorem \ref{x1} to $F$, we obtain that the function
\begin{multline*}
t \mapsto \widehat F(0) -\int_{0}^{t} F(s)\, ds= - T_1R(\omega,A)T_2x -\int_{0}^{t}T_1T(s)A R(\omega, A)T_2x \, ds\\
=-T_1T(t)R(\omega, A)T_2x, \qquad t\ge 0,
\end{multline*}
belongs to $L^p (\mathbb R_+, \wml, X)$, that is, \eqref{cutofflprates} holds.
The proof of \eqref{cutofflp} is similar to the above argument (using Theorem \ref{lpingham} instead of Theorem \ref{x1}), and is omitted.
\end{proof}

Choosing $T_1 = T_2 = I$,
the following corollary of Theorem \ref{energydecayabs} is immediate.

\begin{corollary}\label{energyindiv}
Let $(T(t))_{t \ge 0}$ be a $C_0$-semigroup on a Banach space $X$, with generator $A$, and let $\omega >\max(\omega_0,0)$ be fixed.
Let $x \in X$ be such that
\begin{enumerate}[\rm(i)]
\item $T(\cdot)x \in L^p (\mathbb R_+, X)$ for some $p \ge 1$;
\item $R(\cdot, A)x$ extends analytically to $\Omega_M$,
and its extension $G$ satisfies \eqref{q91}.
\end{enumerate}
Then
 \begin{equation*}
 T(\cdot)R(\omega,A) x \in L^p (\mathbb R_+, w, X)\cap C_0(\mathbb R_+,X),
 \end{equation*}
 where $w$ is as in Theorem \ref{energydecayabs}.
\end{corollary}

Now let $(T(t))_{t\ge0}$ be a $C_0$-semigroup of contractions on a Hilbert space $X$, with generator $A$.  Assume that $D(A^*) = D(A)$, and consider the operator $-(A+A^*)$ with domain $D(A)$.  This operator is symmetric and non-negative, since $A$ is dissipative.  Let $S$ be any non-negative, self-adjoint extension of $-(A+A^*)$ (for example, the Friedrichs extension), and let $B = S^{1/2}$.  Then $D(A) \subset D(B)$ and $B$ is $A$-bounded, since
\begin{multline} \label{BAbd}
\|Bx\|^2 = - ( (A+A^*)x,x ) \le 2 \|Ax\| \, \|x\| \\ \le (\|Ax\| + \|x\|)^2, \,\, x \in D(A).
\end{multline}

\begin{theorem} \label{Hilbert}
Let $(T(t))_{t \ge0}$ be a $C_0$-semigroup of contractions on a Hilbert space $X$, with generator $A$.  Assume that $D(A) = D(A^*)$ and $\sigma(A) \cap i\R$ is empty. Let $M$ be a continuous increasing function such that \eqref{resbd} holds, and let $B$ be as above.  Then
\[
B T(\cdot) A^{-1}x \in L^2(\R_+,w,X) \cap C_0(\R_+,X),  \qquad x \in X,
\]
where $w$ is as in Theorem \ref{energydecayabs}.
\end{theorem}

\begin{proof}  Since $B$ is $A$-bounded, $BA^{-1}T(\cdot)$ is strongly continuous and uniformly bounded.  By Theorem \ref{semigrouprates}, $\lim_{t\to\infty} \|T(t)x\| = 0$ for all $x \in X$, so $BA^{-1}T(\cdot)x \in C_0(\R_+,X)$ for all $x \in X$.

Now, take $x \in D(A)$ with $\|x\|=1$.
Let $F(t) = BT(t)x$.  Then
\[
\|F(t)\|^2 = - ( AT(t)x, T(t)x ) - (A^*T(t)x, T(t)x ) = - \frac{d}{dt} \|T(t)x\|^2.
\]
Hence $F \in L^2(\R_+, X)$. In fact,
\[
\int_0^\infty \|F(t)\|^2 \, dt = 1,
\]
since $\lim_{t\to\infty} \|T(t)x\| = 0$.

Since $B$ is $A$-bounded, it is easy to see that
\[
\widehat F(\lambda) = B R(\lambda,A)x, \qquad \lambda \in \C_+.
\]
Thus $\widehat F$ extends analytically by the same formula to $\rho(A)$ and, by \eqref{BAbd},
\begin{align*}
\|\widehat F(\lambda)\| &\le \|AR(\lambda,A)x\| + \|R(\lambda,A)x\| \\
&= \|\lambda R(\lambda,A)x - x\|+ \|R(\lambda,A)x\| \\
&\le C (1 + |\Im\lambda|) M(|\Im\lambda|), \qquad \lambda \in \Omega_{2M}.
\end{align*}
Moreover
\[
\widehat F(0) - \int_0^t F(s) \, ds = - BT(t) A^{-1} x,
\]
so Theorem \ref{x1} shows that $BT(\cdot)A^{-1}x \in L^2(\R_+,w,X) \cap C_0(\R_+,X)$ and
\begin{equation} \label{BTAest}
\|BT(\cdot) A^{-1}x\|_{L^2(\R_+,w,X)} \le C,
\end{equation}
where the constant $C$ is independent of $x$ when $x \in D(A)$ and $\|x\|=1$.

For $y \in X$ with $\|y\|=1$, we may take a sequence $(x_n)$ in $D(A)$ converging to $y$ with $\|x_n\|=1$.  Using \eqref{BTAest} for $x= x_n$, the boundedness of the operator $BT(t)A^{-1}$, and Fatou's Lemma, we infer that \eqref{BTAest} is also true for $x=y$.
\end{proof}

\begin{example} \label{ex64}
Let $\mathcal{B}$ be a self-adjoint, positive-definite operator on a Hilbert space $\mathcal{H}$, and let $X$ be the Hilbert space $D(\mathcal{B}^{1/2}) \times \mathcal{H}$ with the inner product
\[
\big ( (x_1,x_2), (y_1,y_2) \big)_X = (\mathcal{B}^{1/2}x_1, \mathcal{B}^{1/2}y_1 )_{\mathcal{H}} + (x_2,y_2)_{\mathcal{H}}.
\]
Let $\mathcal{A}$ be an operator such that $D(\mathcal{B}^{1/2}) \subset D(\mathcal{A})$, $D(\mathcal{B}^{1/2}) \subset D(\mathcal{A}^*)$, and $\mathcal{A}$ and $\mathcal{A}^*$ are both bounded in the graph norm on $D(\mathcal{B}^{1/2})$.  Let $A$ be the operator on $X$ defined by
\begin{equation} \label{defA2}
D(A) = D(\mathcal{B}) \times D(\mathcal{B}^{1/2}), \qquad A = \begin{pmatrix} 0 & I \\ -\mathcal{B} & - \mathcal{A} \end{pmatrix}.
\end{equation}
Then a short calculation \cite[Lemma 1, p.74]{Yak99} shows that
\[
D(A^*) = D(\mathcal{B}) \times D(\mathcal{B}^{1/2}), \qquad A^* = \begin{pmatrix} 0 & -I \\ \mathcal{B} & -\mathcal{A}^* \end{pmatrix}.
\]
So
$$
-(A + A^*) = \begin{pmatrix}  0 & 0 \\ 0 & \mathcal{A} + \mathcal{A^*} \end{pmatrix}.
$$
\end{example}

Consider the damped wave equation \eqref{wave} on $\mathcal M$ with
non-empty boundary.
Let $\mathcal{H}
= L^2(\mathcal M)$, and $\mathcal{B} = -\Delta$ with
$D(\mathcal{B}^{1/2}) = H_0^1(\mathcal M)$ and $D(\mathcal{B}) =
H^2(\mathcal M) \cap H_0^1(\mathcal M)$, and $\mathcal{A}$ be the
bounded operator of multiplication by $a$.  These choices fit both
Theorem~\ref{Hilbert} and Example~\ref{ex64} and the operators $A$
defined in \eqref{defA0}--\eqref{defA1} and \eqref{defA2} coincide.

A similar illustration of Theorem~\ref{Hilbert} and Example~\ref{ex64} arises when $\mathcal M$ has no boundary, for example
when $M$ is a torus. Then the boundary condition is omitted from
\eqref{wave}, and $A$, defined by \eqref{defA1} on
 \begin{equation*}
 D(A)=H^2(\mathcal M)\times H^1(\mathcal M),
 \end{equation*}
generates a $C_0$-semigroup of contractions on $X:=H^1(\mathcal M)\times
L^2(\mathcal M)$. In this case
$A$ is not invertible and $0$ is an isolated eigenvalue
corresponding to constant solutions of \eqref{wave}, see
\cite{Le96} and \cite{AnLe12}.  Let $P_0$ be the spectral (Riesz)
projection corresponding to $0$, and let $X_0:= (I-P_0)X$ be equipped
with the inner product $$\langle u, v \rangle_0 =
\langle (-\Delta) u_0, v_0 \rangle + \langle u_1,v_1\rangle, \qquad u=(u_0,u_1),
v=(v_0,v_1) \in X,$$
which is equivalent to the original inner product $\langle \cdot, \cdot \rangle$
on $X$. Then $A_0:=(I-P_0)A$ generates a $C_0$-semigroup of
contractions $(T_{P_0}(t))_{t \ge 0}$ on the Hilbert space $X_0$.
Furthermore,
\begin{equation*}
E(u,t) \sim \|T_{P_0}(t)(I-P_0)u\|^2_0, \qquad u \in X,
\end{equation*}
$A_0$ is invertible in $X_0$ and satisfies the same resolvent estimate
on $i\mathbb R$ as $A$.
Moreover, $A_0$ is given by the matrix \eqref{defA1} restricted to
$X_0$, where $-\Delta$ is now positive definite. See \cite{AnLe12} for the above
properties,
in particular \cite[Part II.4]{AnLe12}. Thus, the study
of energy asymptotics of \eqref{wave} in resolvent terms is
reduced in this case to that for $\mathcal M$ with
boundary considered above.  Both cases are studied thoroughly
in \cite{Le96} and \cite[Parts I.2, II]{AnLe12}, for example.

The following corollary provides a concrete application of Theorem~\ref{Hilbert} and reveals new asymptotic properties of classical
solutions to damped wave equations. Recall that $u$ is said to be
a classical solution of \eqref{wave} on $\mathcal M$ with the boundary $\partial M\ne\emptyset$ if
$$
u \in C(\mathbb R_+,
H^2(\mathcal M) \cap H_0^1(\mathcal M))\cap C^1(\mathbb R_+, H^1_0(\mathcal M)),
$$
and $u$ satisfies \eqref{wave}.  When $\partial M =
\emptyset$ this has to be modified in an obvious way.

\begin{corollary}\label{c65}    Assume that the damped wave equation \eqref{wave} (where
$\partial \mathcal M$ can be empty) is stable at rate $r(t)$.
Then, for each classical solution $u$ of \eqref{wave}, one has
\begin{equation}\label{derivative}
\int_{0}^{\infty} \left|\frac{d}{dt} E(u,t)\right|\, w(t)^2\, dt<
\infty
\end{equation}
where  $M(s) = r^{-1}(c/s)$ for some $c>0$ and $w(t) =
w_{M,\log}(kt)$ for some $k>0$.
\end{corollary}

\begin{proof}
Assume that $\partial M\neq \emptyset$. Let $X=H_0^1(\mathcal
M)\times L^2(\mathcal M)$, and let an operator $A$ on $X$ be
defined by \eqref{defA0} and \eqref{defA1}. Then, as we mentioned in
the introduction,  $A$ generates a $C_0$-semigroup of contractions
in $X$, and $i\mathbb R \subset \rho (A)$.
By our assumption and Theorem~\ref{semigrouprates}
or by
\cite[Theorem 4.4.14]{ABHN01}, there exist positive constants
$c,C$ such that $\|R(is,A)\| \le Cr^{-1}(c/s)$. Now we can apply
Theorem~\ref{Hilbert}, with
$$
B = \begin{pmatrix} 0 & 0 \\ 0 & (2a)^{1/2} \end{pmatrix}, \qquad
x = A(u_0,u_1),
$$
to obtain that
$$ \int_{0}^{\infty} w(t)^2 \int_{\mathcal M} a(x)
|u_t(t; x)|^2 \, dx \, dt < \infty.
$$
It remains to note that
$$
\frac{d}{dt} E(u,t)=-\int_{M} a |u_t(t,x)|^2 \, dx \le 0, \qquad t \ge 0.
$$

If $\partial M=\emptyset$, then the same argument applied to
$A_0$ gives \eqref{derivative}.
\end{proof}

If the set of damping (where $a>0$) has non-empty interior in
$\mathcal M$, then \eqref{wave} is stable at rate $r(t)= 1/\log t$ by \cite{Le96} and \cite{Bu98}, and then $w(t) \sim \log t$.  If the set of
damping  satisfies the so-called geometric control condition
then \eqref{wave} is stable at rate $r(t) =  e^{-\beta t}$
(see \cite{RaTa74} and \cite{BLR92}), and then $w(t) \sim e^{\gamma \sqrt t}$ for some $\gamma >0$.  Note that, according to \cite{Sch10}, if the geometric control
condition fails, then the energy decay can be at an exponential rate for sufficiently smooth initial data $(u_0,u_1)$.  In
many cases where the geometric control condition is
not satisfied, \eqref{wave} is stable at rate $r(t) = t^{-\beta}$
for some $\beta>0$ (see e.g. \cite{AnLe12}).  Then $w(t) \sim (t/\log t)^\beta$.  In
other
cases (see \cite{Ch07},\cite{Ch10}), the stability occurs at rate
$r(t)=e^{-\beta \sqrt t}$ and this gives rise to $w(t) \sim
e^{\gamma t^{1/3}}$ for a fixed $\gamma >0$.

See also \cite{AnLe12} for a recent discussion of these cases and a pertinent study of
polynomial rates when $\mathcal M$ is the torus.

 We note that the operator approach to the study of energy decay in similar settings has a long
history, and there is a vast literature on subject.
A comprehensive account would take too much space here, so we just refer to the recent survey \cite{Alabau} and the papers cited therein.
Among the papers going back to the early days of abstract theory, one could mention \cite{Lagnese}, \cite{Lasiecka} and \cite{Quinn}.

\begin{remark} It is plausible that our results can be applied also to the study of local energy for damped wave equations on exterior domains. See, for example, \cite{AK02}, \cite{Kh03} and \cite{BoRo13} where the operator-theoretic approach played a role. However, this setting seems to require essential extra work in order to put it into our framework, and we do not consider it in this paper.
\end{remark}

\section{Optimality for semigroups}\label{se7}

First we show that while $L^p$-rates for $C_0$-semigroups as in \eqref{hope} are straightforward consequences of their $L^{\infty}$-counterparts, nevertheless the estimate \eqref{hope} is sharp in the case of polynomial rates.

\begin{theorem}\label{xyz}
Given $\alpha > 0$, $p\ge 1$, and a positive function $\gamma\in C_0(\mathbb R_+)$, there exist a Banach space $X_{\alpha}$, a
bounded $C_0$-semigroup \newline\noindent
$(T(t))_{t \ge 0}$ on $X_{\alpha}$ with
generator $A$, and a vector $f\in X_{\alpha}$, such that
\begin{enumerate} [\rm (a)]
\item \label{71a} $\sigma(A) \cap i\R$ is empty and $\|R(is,A)\|_{\mathcal{L}(X_\alpha)}={\rm O}(|s|^{\alpha}), \qquad |s| \to \infty$, \smallskip
\item \label{71b} $\displaystyle \int_0^\infty \|T(t)f\|^p_{X_{\alpha}}dt<\infty$, \smallskip
\item \label{71c} $\displaystyle \int_0^\infty\|T(t)A^{-1}f\|^p_{X_{\alpha}} \Bigl(\frac{t}{\gamma(t)\log (t+2)}\Bigr)^{p/\alpha}
\,dt=\infty$.
\end{enumerate}
\end{theorem}

\begin{proof}
We use the Banach space $X_{\alpha}$, and the semigroup $(T(t))_{t\ge0}$ constructed in \cite[Theorem 4.1]{BoTo10}. To make our proof self-contained, we give the construction here.

Let $(S(t))_{t \ge 0}$ be the left shift semigroup on the space ${\rm BUC}
(\mathbb R_+ )$ of bound\-ed, uniformly continuous, functions on $\R_+$, and let
\begin{align*}
\Omega &:=\{\lambda \in\mathbb C: \Re  \lambda
> - (1+|\Im \lambda|)^{-\alpha}\}, \\
\Omega_0 &:=\Omega \cap \{\lambda \in \mathbb C: |\Re \lambda| < 1 \}.
\end{align*}
Furthermore, let
$X_{\alpha}$ be the space of functions $f\in {\rm BUC}(\R_+ )$
such that the Laplace transform $\widehat{f}$ extends to an
analytic function in $\Omega_0$ (also denoted by $\widehat f$), and
\begin{equation} \label{xalpha}
|\widehat f(\lambda)|(1+|\Im \lambda|)^{-\alpha} \to 0, \qquad
|\lambda| \to \infty, \, \, \lambda \in \Omega_0.
\end{equation}
Then $X_\alpha$ equipped with the norm
\begin{equation*}
\| f \|_{X_{\alpha}}  := \| f \|_{\infty} + \|f \|_\alpha :=
\|f\|_{\infty} + \sup_{\lambda \in \Omega_0} \frac{|\widehat f(\lambda)|}
{(1+|\Im\lambda|)^{\alpha}},
\end{equation*}
is a Banach space.  Let $(T(t))_{t \ge 0}$ be the restriction of $(S(t))_{t \ge 0}$ to $X_\alpha$.  This is  a bounded $C_0$-semigroup on $X_\alpha$. To prove this assertion we start by fixing $t>0$.   For $\lambda \in \C_+$,
\begin{equation} \label{Tthat}
\widehat {T(t)f} (\lambda) =  \int_t^\infty e^{-\lambda (s-t)} f(s) \, ds = e^{\lambda t}\widehat{f}(\lambda) - e^{\lambda t} \int_{0}^{t} e^{-\lambda s}f(s) \, ds.
\end{equation}
Then $\widehat {T(t)f}$ can be extended to  $\Omega_0$ by the same formula.  For $\lambda\in\Omega_0$, we have
$$
|\widehat {T(t)f} (\lambda)|\le e^{t\Re\lambda}|\widehat {f} (\lambda)|
+\int_{0}^{t}e^{(t-s)\Re\lambda}|f(s)| \, ds.
$$
This and \eqref{xalpha} tell us that
$$
|\widehat {T(t)f}(\lambda)|(1+|\Im \lambda|)^{-\alpha} \to 0, \qquad
|\lambda| \to \infty, \, \, \lambda \in \Omega_0,
$$
so $T(t)f \in X_\alpha$.  Furthermore, for $\lambda\in\C_+$ we have
$$
|\widehat {T(t)f} (\lambda)|\, \Re\lambda\le \|f\|_{\infty}.
$$
We obtain from this, together with \eqref{Tthat} for $\lambda\in \Omega_0 \cap \C_-$, that
$$
|\widehat {T(t)f} (\lambda)|\le C\|f\|_{X_\alpha}\big[|\Re \lambda|^{-1}+(1+|\Im \lambda|)^{\alpha}\big], \qquad \lambda\in\Omega_0 \setminus i\R,
$$
for some $C$ which is independent of $t$.
Applying Levinson's $\log$-$\log$ theorem (see, for example, \cite[Section VII.D7]{Koo})
or, rather, its polynomial growth version  \cite[Lemma 4.6.6]{ABHN01},
to $\widehat {T(t)f}$ in the squares
$$
\{\lambda:|\Re \lambda|<(s+2)^{-\alpha},\,
|s-\Im \lambda|<(s+2)^{-\alpha}\},
$$
we conclude that
\begin{equation}\label{XQ8}
\sup_{\lambda\in\Omega_0}
|\widehat {T(t)f} (\lambda)|(1+|\Im \lambda|)^{-\alpha}\le
C\|f\|_{X_\alpha},
\end{equation}
for some $C$ which is independent of $t$.  Thus, $(T(t))_{t \ge 0}$ is a bounded semigroup on $X_\alpha$.  Furthermore, by the definition of $X_\alpha$ and \eqref{Tthat},
$$
\| T(t)f - f \|_\alpha \to 0, \qquad t \to 0+,
$$
and then
$$
\| T(t)f - f \|_{X_\alpha} \to 0, \qquad t \to 0+.
$$
So $(T(t))_{t \ge 0}$ is a bounded $C_0$-semigroup on $X_\alpha$. We define $A$ to be its generator.

To establish the property (\ref{71a}), we have to show that the resolvent $R(\lambda,A)$ satisfies the estimate
\begin{equation}\label{polyn}
\|R(\lambda,A)f\|_{X_\alpha}\le C\big(1+|\Im \lambda|
\big)^\alpha \| f\|_{X_\alpha}, \quad 0 < \Re \lambda<1,\,f \in X_\alpha.
\end{equation}
Fix $\lambda$ with $\Re\lambda \in (0,1)$. For $t\ge0$, we have
$$
\bigl(R(\lambda,A)f \bigr)(t) = \widehat {T(t)f} (\lambda).
$$
By \eqref{XQ8} we conclude that
\begin{equation}\label{esti}
\big \|R(\lambda,A)f \big \|_{\infty}  \le C(1+|\Im \lambda|)^{\alpha} \big( \| f\|_{\infty}+ \| f \|_\alpha \big),
\quad  0 < \Re \lambda<1.
\end{equation}
A simple calculation shows that
\begin{equation*}
\widehat{(R(\lambda,A)f )}(\mu) = -\frac{\widehat f(\lambda) -\widehat f(\mu)}{\lambda - \mu}, \qquad \Re\mu>1.
\end{equation*}
Therefore, $\widehat{R(\lambda,A)f}$ extends analytically to $\Omega_0$, and
$$
\widehat{(R(\lambda,A)f )}(\mu)=\begin{cases} -\frac{\widehat
f(\lambda) -\widehat f(\mu)}{\lambda - \mu}, \qquad \lambda \neq
\mu,\, \mu \in \Omega_0,\\
-{\widehat f\,}^\prime(\mu), \qquad \qquad \lambda=\mu.
\end{cases}
$$

To estimate $\|R(\lambda,A)f\|_\alpha$, take $\mu \in \Omega_0$.  If $|\lambda-\mu|\ge 1$, then
$$
|\widehat{(R(\lambda,A)f)}(\mu)|\le|\widehat{f}(\lambda)|+|\widehat{f}(\mu)|
\le C(1+|\Im \mu|)^{\alpha}(1+|\Im \lambda|)^{\alpha}\|f\|_\alpha.
$$
Otherwise $1+|\Im\lambda|^\alpha$ and $1+|\Im\mu|^\alpha$ are comparable.  If
$$
1>|\lambda-\mu|\ge \frac{1}{2(1+|\Im \lambda|)^\alpha},\quad \mu\in \Omega_0,
$$
then we have
\begin{eqnarray*}
|\widehat{(R(\lambda,A)f )}(\mu)| &\le&
2(1+|\Im \lambda|)^{\alpha} \bigl(|\widehat{f}(\lambda)|+|\widehat{f}(\mu)| \bigr)\\ &\le&
c(1+|\Im \lambda|)^\alpha(1+|\Im \mu|)^\alpha \|f\|_\alpha.
\end{eqnarray*}
Finally, if
$$
|\lambda-\mu|\le\frac{1}{2(1+|\Im \lambda|)^\alpha},
$$
then, applying Cauchy's formula on the circle
$
\{z\in \mathbb C:|z-\lambda|=\frac 23(1+|\Im \lambda|)^{-\alpha}\},
$
we obtain that
$$
|\widehat{(R(\lambda,A)f)}(\mu)|\le
C(1+|\Im \mu|)^{\alpha} (1+|\Im \lambda|)^{\alpha}\|f\|_\alpha.
$$
Thus,
\begin{equation}\label{esto}
\big\|R(\lambda,A)f \big \|_\alpha  \le C(1+|\lambda|)^{\alpha}  \| f \|_\alpha, \qquad 0<\Re \lambda<1.
\end{equation}

The estimates \eqref{esti} and \eqref{esto} together give us
\eqref{polyn}, and it follows that (\ref{71a}) holds.  Moreover, Theorem \ref{semigrouprates} shows that
\begin{equation} \label{integral}
A^{-1}f = \lim_{t\to \infty} \big(A^{-1}f - T(t)A^{-1}f \big)=-\int_{0}^{\infty} T(t)f \, dt,
\end{equation}
for every $f \in X_\alpha$, where the integral may be improper.

To construct $f \in X_\alpha$ satisfying (\ref{71b}) and (\ref{71c}),  we can assume that $\gamma$ is increasing and that $\gamma(t) \ge t^{-1}$ for all $t\ge1$.  We shall use Proposition \ref{prop}. For a complex measure $\mu$ on $\C \setminus \Omega$ with compact support, let $\XL\mu$,  $\mathcal G\mu$ and $\mathcal N\mu$ be defined by \eqref{Lmu}--\eqref{Nmu}.  For $f = \mathcal{L}\mu$, the integral in \eqref{integral} is absolutely convergent and simple calculations show that
\[
\widehat{T(t)\XL\mu}(z) = \mathcal G\mu(t,z),  \qquad (A^{-1}\XL\mu)(t) =  \mathcal N\mu(t).
\]

By Proposition \ref{prop}, there exist $\{k_n : n \in \N\}\subset \mathbb N$, and complex measures $\mu_n$ with compact support in $\C \setminus \Omega$ such that $k_n \to \infty$ as $n \to \infty$, and, for each $n\ge1$,
\begin{enumerate}[(i)]
\item  by \eqref{X3} and \eqref{XQ4}, $f_n := \XL\mu_n \in X_\alpha$, and
\begin{align*}
\|T(t)f_n\|_{X_\alpha}=\,& \|T(t)\XL\mu_n\|_{\infty}+\|\widehat{T(t)\XL \mu_n}\|_{\alpha}
\\
=\,& \|T(t)\XL\mu_n\|_{\infty}+\|\mathcal G\mu_n(t,\cdot)\|_{\alpha}
\\
\le\,&
c\chi_{\{t:t\le 2k_n\}}+ ce^{-\rho t},
\end{align*}
so
\[
\|T(\cdot)f_n\|_{L^p(\R_+,X_\alpha)}^p = \int_0^\infty \|T(t)f_n\|_{X_\alpha}^p \, dt \le ck_n;
\]
\item by \eqref{X5}, for $0\le t \le k_n$,
\begin{align*}
\|T(t)A^{-1}f_n\|_{\infty} &\ge |(T(t)A^{-1}f_n)(k_n-t)| \\
&= |(A^{-1}f_n)(k_n)|=|\XN\mu_n (k_n)| \ge c_1\Bigl(\frac{\log k_n}{k_n}\Bigr)^{1/\alpha},
\end{align*}
\item by \eqref{X6},
$$
\|T(t)A^{-1}f_n\|_{\infty}\le\, c_2\chi_{\{t:t\le 2k_n\}}\Bigl(\frac{\log k_n}{k_n}\Bigr)^{1/\alpha}+
c_2e^{-\rho t}.
$$
\end{enumerate}
Passing to a subsequence and relabelling, we may assume that
$$
k_1 \ge 3, \quad  k_n \ge \max \left(3, \left(\frac{2c_2}{c_1}\right)^p \right) k_{n-1} \,\, (n\ge2), \quad 2^n \gamma(2k_n/3)^{1/\alpha} \to 0.
$$
Let
$$
f=\sum_{n=1}^\infty 2^{-n}k_{n}^{-1/p}f_{n}.
$$
Then
\[
\|T(\cdot)f\|_{L^p(\R_+,X_\alpha)} \le \sum_{n=0}^\infty 2^{-n} k_{n}^{-1/p} \|T(\cdot)f_{n}\|_{L^p(\R_+,X_\alpha)} < \infty.
\]
Thus (\ref{71b}) holds.

For $2k_{n}/3 \le t \le k_{n}$, we have
\begin{align*}
\lefteqn{\hskip-20pt\|T(t)A^{-1}f\|_\infty} \\
 &\ge 2^{-n}k_{n}^{-1/p} \|T(t)A^{-1}f_{n}\|_\infty - \sum_{m \ne n} 2^{-m}k_{m}^{-1/p} \|T(t)A^{-1}f_{m}\|_\infty \\
&\ge c_1  2^{-n}k_{n}^{-1/p} \left( \frac{ \log k_{n}}{k_{n}} \right)^{1/\alpha} \\
&\phantom{XX} - c_2 \sum_{m=n+1}^\infty  2^{-m} k_{m}^{-1/p} \left( \frac{ \log k_{m}}{k_{m}} \right)^{1/\alpha} - c_2 \sum_{m=1}^{n-1}  2^{-m} k_{m}^{-1/p} e^{-\rho t}  \\
&\ge c_1  2^{-n}k_{n}^{-1/p} \left( \frac{ \log k_{n}}{k_{n}} \right)^{1/\alpha} - c_2 2^{-n} k_{n+1}^{-1/p} \left( \frac{ \log k_{n}}{k_{n}} \right)^{1/\alpha} -c_2 e^{-\rho t} \\
&\ge  c_1  2^{-(n+1)}k_{n}^{-1/p} \left( \frac{ \log k_{n}}{k_{n}} \right)^{1/\alpha} - c_2 e^{-\rho t}.
\end{align*}
Hence
\begin{align*}
\lefteqn{\hskip-20pt\int_{2k_{n}/3}^{k_{n}} \left( \|T(t)A^{-1}f\|_{X_{\alpha}} + c_2e^{-\rho t} \right)^p \Bigl(\frac{t}{\gamma(t)\log (t+2)}\Bigr)^{p/\alpha}
\,dt}  \\
&\ge \frac{k_{n}}{3} c_1^p 2^{-p(n+1)} k_{n}^{-1} \left( \frac{ \log k_{n}}{k_{n}} \right)^{p/\alpha} \left( \frac{2k_{n}/3}{\gamma(2k_n/3) \log(k_{n} + 2)} \right)^{p/\alpha}\\
&\to\infty.
\end{align*}
Since $\gamma(t) \ge t^{-1}$ for $t\ge1$, we have that $t \mapsto e^{-\rho t} \Bigl(\frac{t}{\gamma(t)\log (t+2)}\Bigr)^{p/\alpha}$ is in $L^p(\R_+)$, and (\ref{71c}) follows.
\end{proof}

Analogously, for logarithmic rates, using Proposition~\ref{propA} we obtain:

\begin{theorem}\label{xyzA}
Given $\alpha > 0$ and $p\ge 1$, there exist $\gamma>0$ and a Banach space $X_{\alpha}$, a bounded $C_0$-semigroup
$(T(t))_{t \ge 0}$ on $X_{\alpha}$ with
generator $A$, and a vector $f\in X_{\alpha}$, such that
\begin{enumerate} [\rm (a)]
\item  $\|R(is,A)\|_{\mathcal{L}(X_\alpha)}={\rm O}\big((\log |s|)^{\alpha}\big), \qquad |s| \to \infty$, \smallskip
\item $\displaystyle \int_0^\infty \|T(t)f\|^p_{X_{\alpha}}dt<\infty$, \smallskip
\item $\displaystyle \int_0^\infty\|T(t)A^{-1}f\|^p_{X_{\alpha}} e^{\gamma t^{1/(\alpha+1)}}
\,dt=\infty$.
\end{enumerate}
\end{theorem}

Next, we show that Corollary \ref{energyindiv} and hence Theorem \ref{energydecayabs} are optimal even on Hilbert spaces.
Consider the left shift $C_0$-semigroup $(S(t))_{t \ge 0}$ on the Hilbert space $L^2(\mathbb R_+)$, with generator $A$. Given $\alpha>0$, we again let $\Omega=\{\lambda \in\mathbb C: \Re \lambda > -1/(1+|\Im\lambda|)^\alpha\}$.

\begin{theorem}\label{t7}
Given $\alpha > 0$, $p\ge1$, and a positive function $\gamma
\in C_0(\mathbb R_+)$, there exists $f\in X:=L^2(\mathbb R_+)$ such that the following hold:
\begin{enumerate}[\rm (a)]
 \item \label{73a} $\int_0^\infty\|S(t)f\|_X^p\,dt<\infty$,
\smallskip
\item \label{73b} $R(\cdot,A)f$ extends analytically to $\Omega$, and its extension $G$ satisfies
$$
\|G(\lambda)\|_X \le C (1+|\Im \lambda|)^{\alpha/2}, \qquad \lambda\in\Omega,
$$
\item \label{73c} $\displaystyle
 \int_0^\infty \|S(t)A^{-1}f\|_X^p \, \Bigl(\frac{t}{\gamma(t)\log (t+2)}\Bigr)^{p/\alpha}dt = \infty$.
\end{enumerate}
\end{theorem}

\begin{proof} The proof is similar to those of Theorems \ref{example} and \ref{xyz}, and we omit some details.  We use Proposition~\ref{prop}, for $\alpha/2<\beta<\alpha/2+\alpha/p$, to take a strictly increasing sequence $\{k_n: n\ge1 \}\subset \mathbb N$, complex measures $\mu_n$ with compact support in $\C \setminus \Omega$, and  $\{w_n\}\subset  \mathbb C\setminus \Omega$, such that
\begin{equation} \label{wkn}
\lim_{n\to\infty} \frac{\log|w_n|}{\log k_n}=\frac 1\alpha, \qquad \supp\mu_n\subset \{z:|z-w_n|<1\},
\end{equation}
 and properties \eqref{X3}, \eqref{XQ4},  \eqref{X5} and \eqref{X6} hold for $\mu=\mu_n$ and $k=k_n$.  Let $f_n:=  \XL\mu_n$.  Using \eqref{X3} similarly to \eqref{T11}, we have that $f_n \in L^2(\mathbb R_+)$ and
\begin{equation} \label{T1}
 \|S(t)f_n\|_X=\|S(t)\mathcal L\mu_n\|_X\le c k_n^{1/4}\chi_{\{t:t\le 2k_n\}}+ C e^{-\rho t}, \qquad t\ge 0.
\end{equation}
We have $R(\lambda,A)f_n=\mathcal G\mu_n(\cdot,\lambda)$ for $\Re\lambda>0$. Putting $G_n(\lambda)(t) =  \mathcal G\mu_n(t,\lambda)$ for $\lambda \in \Omega$, we  obtain an analytic extension to a function $G_n : \Omega \to L^2(\R_+)$.   For $\lambda \in \Omega$, \eqref{XQ4}, \eqref{wkn}, and the fact that $\beta < \alpha/2+\alpha/p$, give
\begin{align} \label{T4}
\|G_n(\lambda)\|_X &\le Ck_n^{1/4}\big(|\Im \lambda|^{\beta}\chi_{\{\lambda:|\lambda-w_n|<2\}} +1 \big) + C \\
&\le C k_n^{1/4+1/p}(1+|\Im \lambda|)^{\alpha/2}. \nonumber
\end{align}
Also, $A^{-1}f_n=\XN\mu_n$, and
$$
\|S(t)A^{-1}f_n\|_X^2=\int_0^\infty |\XN\mu_n(s+t)|^2ds.
$$
By \eqref{X5},
\begin{equation}
\|S(t)A^{-1}f_n\|_X\ge ck_n^{1/4}\Bigl(\frac{\log k_n}{k_n}\Bigr)^{1/\alpha},\qquad t\le k_n,\label{T5}
\end{equation}
and by \eqref{X6},
\begin{equation}
\|S(t)A^{-1}f_n\|_X\le Ck_n^{1/4}\Bigl(\frac{\log k_n}{k_n}\Bigr)^{1/\alpha}\chi_{\{t:t\le 2k_n\}}+ Ce^{-\rho t},\quad t\ge 0.\label{T6}
\end{equation}

The remainder of the proof closely follows the corresponding part of the proof of Theorem \ref{xyz}.  After passing to a suitable subsequence and relabelling, we let
$$
f=\sum_{n\ge 0}2^{-n}k_{n}^{-(1/p+1/4)}f_{n}.
$$
Then (\ref{73a}) follows from \eqref{T1}, (\ref{73b}) follows from  \eqref{T4}, and, finally, (\ref{73c}) follows from \eqref{T5} and \eqref{T6} if the subsequence increases sufficiently fast.
\end{proof}

\section{Acknowledgements}

The authors would like to thank the referee for
useful comments.

\end{document}